\documentclass[10pt,reqno]{amsart}

\usepackage{amsmath}
\usepackage{amsfonts}
\usepackage{amssymb}
\usepackage{amsthm}
\usepackage{mathtools}

\usepackage{stmaryrd}

\usepackage{mathrsfs}
\usepackage{latexsym}
\usepackage{verbatim} 
\numberwithin{equation}{section}
\usepackage{enumitem}

\usepackage{graphicx}
\usepackage{subfigure}
\usepackage{epstopdf}

\usepackage{empheq}

\usepackage{ifthen} 

\provideboolean{shownotes} 
\setboolean{shownotes}{true} 
\newcommand{\margnote}[1]{
\ifthenelse{\boolean{shownotes}}%
{\marginpar{\raggedright\tiny\texttt{#1}}}%
{}%
}
\newcommand{\hole}[1]{
\ifthenelse{\boolean{shownotes}}%
{\begin{center} \fbox{ \rule {.25cm}{0cm}
\rule[-.1cm]{0cm}{.4cm} \parbox{.85\textwidth}{\begin{center}
\texttt{#1}\end{center}} \rule {.25cm}{0cm}}\end{center}}
{}
}

\theoremstyle{plain}

\newtheorem{lemma}{Lemma}[section]
\newtheorem{theorem}[lemma]{Theorem}
\newtheorem{proposition}[lemma]{Proposition}
\newtheorem{corollary}[lemma]{Corollary}

\newtheorem*{stheorem}{Theorem}

\theoremstyle{definition}

\newtheorem{remark}[lemma]{Remark}
\newtheorem{definition}[lemma]{Definition}

\newtheorem{example}[lemma]{Example}

\theoremstyle{remark}

\usepackage{cite} 

\usepackage[colorlinks=true,urlcolor=blue,
citecolor=red,linkcolor=blue,linktocpage,pdfpagelabels,
bookmarksnumbered,bookmarksopen]{hyperref}

\usepackage{orcidlink}

\usepackage{cleveref}
\usepackage[pagewise,mathlines]{lineno}

\newcommand{\Id}{\mathrm{Id}}

\newcommand{\bw}{\mathbf{w}}
\newcommand{\bA}{\mathbf{A}}

\DeclareFontFamily{U}{mathx}{}
\DeclareFontShape{U}{mathx}{m}{n}{<-> mathx10}{}
\DeclareSymbolFont{mathx}{U}{mathx}{m}{n}
\DeclareMathAccent{\widehat}{0}{mathx}{"70}
\DeclareMathAccent{\widecheck}{0}{mathx}{"71}

\newcommand{\R}{\mathbb{R}}
\newcommand{\C}{\mathbb{C}}

\newcommand{\N}{\mathbb{N}}

\newcommand{\vep}{\varepsilon}

\newcommand{\cT}{{\mathcal{T}}}

\newcommand{\cL}{{\mathcal{L}}}

\newcommand{\cD}{{\mathcal{D}}}

\newcommand{\cR}{{\mathcal{R}}}
\newcommand{\cS}{{\mathcal{S}}}

\newcommand{\cP}{{\mathcal{P}}}

\newcommand{\cA}{{\mathcal{A}}}

\newcommand{\ccC}{\mathscr{C}}
\newcommand{\ccB}{\mathscr{B}}

\renewcommand{\Re}{\mathrm{Re}\,} 
\renewcommand{\Im}{\mathrm{Im}\,}

\newcommand{\ind}{\mathrm{ind}\,}

\newcommand{\nul}{\mathrm{nul}\,}

\usepackage{scalerel}

\newcommand{\ep}{\epsilon}

\newcommand{\ess}{\sigma_\mathrm{\tiny{ess}}}
\newcommand{\ptsp}{\sigma_\mathrm{\tiny{pt}}}

\newcommand{\iptsp}{\widetilde{\sigma}_\mathrm{\tiny{pt}}}
\newcommand{\sppi}{\sigma_\pi}
\newcommand{\spd}{\sigma_\delta}

\newcommand{\<}{\langle}
\renewcommand{\>}{\rangle}

\newcommand{\ShowColoredChanges}{true} 
\ifthenelse{\isundefined{\ShowColoredChanges}}
  {

  }
  {

  }

\begin{document}

\title[Stability of diffusion-degenerate Nagumo fronts I]{Stability of stationary reaction diffusion-degenerate Nagumo fronts I: \\spectral analysis}

\author[R. Folino]{Raffaele Folino \orcidlink{0000-0001-9089-1393}}
 
\address{{\rm (R. Folino)} Instituto de 
Investigaciones en Matem\'aticas Aplicadas y en Sistemas\\Universidad Nacional Aut\'onoma de 
M\'exico\\Circuito Escolar s/n, Ciudad de M\'{e}xico C.P. 04510 (Mexico)}

\email{folino@aries.iimas.unam.mx}

\author[C. A. Hern\'{a}ndez Melo]{C\'{e}sar A. Hern\'{a}ndez Melo \orcidlink{0009-0003-7194-4543}}

\address{{\rm (C. A. Hern\'{a}ndez Melo)} Departamento de Matem\'{a}tica\\Universidade Estadual de Maring\'{a}\\Av. Colombo, 5790 Jd. Universit\'ario, CEP 87020-900, Maring\'a, PR (Brazil)}

\email{cahmelo@uem.br}

\author[L. F. L\'{o}pez R\'{\i}os]{Luis F. L\'{o}pez R\'{\i}os \orcidlink{0000-0001-9949-2620}}
 
\address{{\rm (L. F. L\'{o}pez R\'{\i}os)} Instituto de 
Investigaciones en Matem\'aticas Aplicadas y en Sistemas\\Universidad Nacional Aut\'onoma de 
M\'exico\\Circuito Escolar s/n, Ciudad de M\'{e}xico C.P. 04510 (Mexico)}

\email{luis.lopez@aries.iimas.unam.mx}

\author[R. G. Plaza]{Ram\'on G. Plaza \orcidlink{0000-0001-8293-0006}}

\address{{\rm (R. G. Plaza)} Instituto de Investigaciones en Matem\'aticas Aplicadas y en Sistemas\\Universidad Nacional Aut\'onoma de M\'exico\\Circuito Escolar s/n, Ciudad Universitaria\\C.P. 04510 Cd. de M\'{e}xico (Mexico)}

\email{plaza@aries.iimas.unam.mx}

\begin{abstract}
This paper establishes the spectral stability of monotone, stationary front solutions for reaction-diffusion equations where the reaction function is of Nagumo (or bistable) type and with diffusion coefficients which are density dependent and degenerate at zero (one of the equilibrium points of the reaction). These stationary profiles connect the non-degenerate equilibrium point with the degenerate state at zero, they are monotone, and arrive to the degenerate state at a finite point. They are neither sharp nor smooth. The degeneracy of the diffusion precludes the application of standard techniques to locate the essential spectrum of the linearized operator around the wave in the energy space $L^2$. This difficulty is overcome with a suitable partition of the spectrum, the analysis of singular sequences, a generalized convergence of operators technique and refined energy estimates. It is shown that the $L^2$-spectrum of the linearized operator around the front is real and with a spectral gap, that is, a positive distance between the imaginary axis and the rest of the spectrum, with the exception of the origin. Moreover, the origin is a simple isolated eigenvalue, associated to the derivative of the profile as eigenfunction (the translation eigenvalue). Finally, it is shown that the linearization generates an analytic semigroup that decays exponentially outside a one-dimensional eigenspace associated to the zero eigenvalue.
\end{abstract}



\maketitle
\setcounter{tocdepth}{1}



\section{Introduction}
\label{secintro}

In this paper we consider the following reaction-diffusion equation
\begin{equation}
\label{degRD}
 	u_t = (D(u)u_x)_x + f(u),
\end{equation}
for an unknown density $u= u(x,t)\in\R$, with $x \in \R$, $t > 0$, where the diffusion coefficient $D(u)$ is a non-negative function depending on $u$ and \textit{degenerate} at $u = 0$. To be more precise, we assume that 
$D$ satisfies 
\begin{equation}
\label{hypD}
\begin{aligned}
& D(0) = 0, \;\; \, D(u) > 0 \; \; \text{for all} \, u \in (0,1],\\
& D \in C^2([0,1];\R) \;\; \text{with} \; D'(u) > 0 \; \text{for all} \; u \in [0,1].
\end{aligned}
\end{equation}
As an example we have the quadratic function 
\begin{equation}
\label{Dbeta}
 D(u) = u^2 + b u,
\end{equation}
for some constant $b > 0$, which was proposed by Shigesada \textit{et al.} \cite{SKT79} to model dispersive forces 
due to mutual interferences between individuals of an animal population.

The reaction term $f \in C^2([0,1];\R)$ is supposed to be of \textit{bistable or Nagumo type} 
\cite{NAY62,McKe70}, that is, the following conditions are satisfied:
\begin{equation}
\label{bistablef}
	\begin{aligned}
	&f(0)=f(\alpha)=f(1)=0,
		&\qquad &f'(0), f'(1)<0,\quad f'(\alpha)>0,\\
	&f(u)>0\textrm{ in }(\alpha,1),
		&\qquad &f(u)<0\textrm{ in }(0,\alpha),
	\end{aligned}
\end{equation}
for a certain $\alpha \in (0,1)$. A typical reaction function satisfying \eqref{bistablef}, which is often found in
the literature, is the cubic polynomial
\begin{equation}
\label{cubicf}
	f(u)= u(1-u)(u-\alpha).
\end{equation}
Reaction functions of bistable type are used in many models of natural phenomena, such as electrothermal instability \cite{I95} or nerve conduction \cite{Lbr67a,McKe70}, among others. In terms of continuous descriptions of the spread of biological populations, it is often used to model the \emph{strong Allee effect} (cf. Murray \cite{MurI3ed}), namely, kinetics exhibiting positive growth rate for population densities over a threshold value ($u > \alpha$) and decay for densities below such value ($u < \alpha$). The function $f$, which has two competing stable states, $u=0$ and $u=1$, can also be interpreted as the derivative of a double-well potential, $F(u) = - \int_0^u f(s) \, ds$, with wells centered at such states.

The reaction-diffusion equation \eqref{degRD} with constant diffusivity, $D(u) \equiv D > 0$, and with the cubic nonlinearity \eqref{cubicf}, has been applied to many different physical and biological models and there is no consensus about its nomenclature. It is known as the Allen-Cahn equation to describe the motion of boundaries between phases in alloys \cite{AlCa79}, the Nagumo equation in neurophysiological modeling \cite{McKe70,NAY62}, the bistable reaction-diffusion equation \cite{FiM77}, the Chafee-Infante equation \cite{ChIn74b,ChIn74a} as well as the real Ginzburg--Landau equation for the variational description of phase transitions \cite{MeSc04a}. For simplicity, we follow McKean \cite{McKe70} and we call it the \emph{Nagumo reaction-diffusion equation}. 

\subsection*{Degenerate diffusivity and traveling fronts}

Since the seminal works by Kolmogorov, Petrovsky and Piskunov (KPP) \cite{KPP37} and Fisher \cite{Fis37}, reaction-diffusion equations of the form \eqref{degRD} have constituted prototypical models of multiple phenomena such as population dynamics, porous media, chemical reactions with diffusion and action potential propagation, among others. In all these contexts, traveling front solutions play a prominent role. There is a vast literature on the analysis of reaction-diffusion equations and their traveling front solutions which we will not review here. In early works the diffusion coefficient (or diffusivity) was considered as a positive constant. Since the work of Skellam \cite{Ske51} (see also \cite{Aron85,OkLe01}), who stressed out the role of the underlying random process prior to the diffusion approximation, it is now clear that in many situations in physics and biology the diffusivity might be a function of the density. 
The resulting nonlinear diffusion coefficient is called \emph{degenerate} if it approaches zero when the density tends to certain equilibrium points of the reaction (typically zero). In terms of population dynamics, zero diffusion for zero densities is tantamount to requiring no motility in regions of space where the population is very scarce or near absent; this is an experimentally observed feature that appears, for example, in models of bacterial dynamics \cite{KMMUS,LMP1,SMGA01}. In this paper, under hypotheses \eqref{hypD}, we assume that the nonlinear diffusion function $D = D(u)$ vanishes only at $u = 0$; in particular, the condition that $D > 0$ for $u > 0$ means that populations tend to ``avoid crowds" (cf. Aronson \cite{Aron85}). Models with degenerate diffusivities are also endowed with interesting mathematical properties. Among the new features one finds the emergence of traveling waves of ``sharp" type \cite{SaMa94a,SaMa97,Sh10} and, notably, that solutions may exhibit finite speed of propagation of initial disturbances, in contrast with the strictly parabolic case \cite{GiKe96}. Clearly, not all models with degenerate diffusions are related to biology. There is an important example coming from chemical engineering: the very well-known choice $D(u) = m u^{m-1}$, $m \geq 1$, which is often used to model diffusion in porous media; see, e.g., \cite{Muskat37,Vaz07} and the references therein. 

Reaction-diffusion models are typically formulated to support the emergence of traveling fronts, which are solutions of the form
\[
u(x,t) = \varphi(x - ct),
\]
where $\varphi : \R \to \R$ is known as the wave profile function and $c \in \R$ is the speed of propagation. Such fronts have asymptotic limits, $u_\pm = \lim_{\xi \to \pm \infty} \varphi(\xi)$, which are equilibrium points of the reaction function under consideration, $f(u_\pm) = 0$. Traveling fronts model a vast set of phenomena, such as invasions of biological agents or envelopes of bacterial colonies, for instance, and their study is of fundamental interest. In the case of degenerate diffusion coefficients, the existence of fronts was first studied for very specific forms of the nonlinear diffusion function, see \cite{Aron85,ARR81,New80,New83} for an abridged list of references, particularly in the case of a diffusivity of porous medium type. In the case of generic degenerate diffusion functions satisfying hypotheses \eqref{hypD}, Sanchez-Gardu\~no and Maini proved the existence of traveling fronts for kinetics of Fisher-KPP (also known as monostable) type \cite{SaMa94a,SaMa95} and of Nagumo (bistable) type \cite{SaMa97}. The authors apply a dynamical systems approach to establish the existence of heteroclinic connections in both cases. The degeneracy, for instance, is responsible for the emergence of \emph{sharp} traveling front solutions of the form $u(x,t) = \varphi_{c_*}(x-c_* t)$, where $c_* > 0$ is a threshold (sharp) speed value and the profiles satisfy $\phi_{c_*}(\xi) \equiv 0$ for all $\xi \geq \xi_0$ for a certain finite value $\xi_0 \in \R$. These sharp fronts are monotone but not smooth. In other words, these sharp orbits connect the two equilibrium points of the reaction (namely, $u = 0$ and $u = 1$) and arrive at the degenerate state at a finite point; the connection is continuous but not of class $C^1$, and hence the word ``sharp" to describe them. Sharp fronts occur for both the bistable and the Fisher-KPP cases. In addition, the authors in \cite{SaMa97,SaMa95,SaMa94a} also show the existence of families of smooth traveling fronts with speeds strictly greater than the threshold value $c_*$. In the Nagumo case, the dynamics is much richer, mainly due to the presence of a third (unstable) equilibrium point $u = \alpha$. This leads to a wider range of possible homoclinic and heteroclinic connections compared to those of the Fisher-KPP case. The authors in \cite{SaMa97} not only prove the existence of a unique positive wave speed, $c_* > 0$, associated to fronts of sharp type: depending on conditions relating $D$ with $f$, there exists a continuum of monotone fronts, pulses and oscillatory waves. An interesting structure also emerges in the Nagumo case: if $\int_0^1 D(u) f(u) \, du = 0$ then there exists a unique stationary monotone increasing front (with speed $c = 0$) connecting $u = 0$ with $u = 1$, as well as a unique stationary, monotone decreasing front connecting $u = 1$ with $u = 0$. These are \emph{stationary} degenerate Nagumo fronts, which constitute the main subject of this paper.

\subsection*{Previous results and literature review} In the literature there are many articles that describe the ``attractiveness'' property of sharp fronts for solutions to the Cauchy problem for reaction diffusion-degenerate equations with sufficiently fast decaying initial data, such as compactly supported (see, e.g.,\cite{Bir97,Bir02}), exponentially decaying \cite{KaRo04b,KaRo04a,DiKa12}, or even semi-compactly supported initial data \cite{XJMY24,XJMY25}. They prove that solutions to the Cauchy problem for \eqref{degRD} with initial data in decaying classes, such as compactly or semi-compactly supported, or exponentially decaying, evolve in time towards the degenerate front with minimum speed $c_* > 0$, that is, towards the sharp-front wave. In these analyses, the authors usually assume that the initial condition is non-negative. Even though some of these works contain the word ``stability" in the title, they do not truly analyze the stability of the fronts under small perturbations, as the initial conditions decay to zero at $+\infty$ or at $-\infty$, whereas the fronts have finite, non-zero limits. In other words, true perturbations of fronts may change their sign and the associated initial conditions, $u_0(x) = \varphi(x) + v_0(x)$, where $\varphi$ is the front profile and $v_0$ denotes a perturbation of the front, definitely do not have compact nor semi-compact support.

It is important to mention that the literature on the \emph{dynamical} stability of degenerate fronts under small perturbations is very scarce. Here we enlist the only works which, up to our knowledge, are devoted to the stability of degenerate fronts:
\begin{itemize}
\item Hosono's 1986 paper \cite{Hos86} is perhaps the first work that addresses the rigorous stability analysis of a sharp traveling front (with speed $c = c_*$) for a degenerate-diffusion equation with a porous-medium type of diffusion coefficient, that is, $u_t = (u^m)_{xx} + f(u)$, with $m > 0$ and reaction function $f$ of Nagumo type. The method of proof is  based on the construction of super- and sub-solutions and the comparison  principle. Hosono establishes the asymptotic convergence of solutions to the nonlinear equation to a  translated front when the initial data is close to the sharp front profile. 
\item The paper by Leyva and Plaza \cite{LeP20} establishes the spectral stability of smooth traveling fronts, with speed $c > c_*$, for diffusion-degenerate equations of Fisher-KPP type studied by S\'anchez-Gardu\~no and Maini \cite{SaMa95}. Here the degenerate diffusion coefficients are of generic type, satisfying \eqref{hypD}. In their work, the diﬃculties associated with the degeneracy of the diﬀusion term are overcome with the derivation of a kind relative entropy estimate with a well-suited exponential weight and by a suitable partition of the spectrum, tailored for degenerate problems.
\item The follow-up paper by Leyva \emph{et al.} \cite{LeLoP22} extends the previous spectral theory from \cite{LeP20} to the bistable or Nagumo case. The authors study the smooth, traveling fronts (with $c > c_*$) discovered by S\'anchez-Gardu\~no and Maini \cite{SaMa97} for the Nagumo reaction function. Once again, the diffusivity coefficient is of generic type. The authors use the same techniques to show the spectral stability of smooth degenerate Nagumo fronts in exponentially weighted Sobolev spaces.
  \item The recent analysis by Dalibard \emph{et al.} \cite{DaLoPe24} contains the first nonlinear stability result of smooth traveling fronts for equations with diffusivity of porous medium type and a generalized reaction of Fisher-KPP type. The authors show that that the linearized system underlies a spectral gap property in weighted $L^2$ spaces and provide quantitative estimates on the rate of decay of solutions. The nonlinear terms are controlled in the $L^\infty$ norm using the maximum principle. The authors prove that these traveling waves are nonlinearly stable under small
perturbations.
\end{itemize} 
The aforementioned articles \cite{DaLoPe24,Hos86,LeP20,LeLoP22} are the only ones that, as far as we know, do prove some form of stability -- spectral or nonlinear -- of the degenerate fronts under small perturbations. The remaining cited literature \cite{Bir97,KaRo04b,KaRo04a,DiKa12,Bir02,XJMY24,XJMY25} focuses on the complementary problem of ``attractiveness''.

\subsection*{The contributions of this paper}
This work is the first of a series of two papers that study the stability of stationary reaction diffusion-degenerate Nagumo fronts. This work establishes, in particular, the first step of a general stability program: the property of \emph{spectral stability}, which is based on the analysis of the spectrum of the linearized differential operator around the wave. Our approach is linked to the modern stability theory of nonlinear waves based on Evans function methods and the relation between spectral stability properties with the nonlinear stability of the wave as a solution to the evolution equation under small perturbations. We refer the reader to the papers by Alexander, Gardner and Jones \cite{AGJ90}, Sandstede \cite{San02}, the book by Kapitula and Promislow \cite{KaPro13} and the multiple references therein. Spectral stability can be formally defined as the property that the linearized operator around the traveling front, posed in an appropriate energy space, is ``well-behaved" in the sense that there are no eigenvalues with positive real part which could render exponentially growing-in-time solutions to the linear problem; for the precise statement, see Definition \ref{defspecstab} below. The degeneracy of the diffusion coefficient, however, is responsible of some technical difficulties even at the spectral level. Since the diffusion coefficient vanishes at $u=0$, one of the end points of the front, the corresponding asymptotic coefficient matrix (when the spectral problem is written as a first order system) is no longer hyperbolic; this prevents the direct application of standard results linking hyperbolicity, exponential dichotomies and Fredholm borders location to locate the essential spectrum. To overcome this, we adopt the particular partition of spectrum proposed by Leyva \emph{et al.} \cite{LeP20, LeLoP22} to deal with degenerate problems. Using these definitions and the techniques employed in \cite{LeP20, LeLoP22} -- parabolic regularization to locate a subset of the compression spectrum, the existence of Weyl sequences and energy estimates to control a subset of the approximate spectrum and a key identity combined with energy estimates to locate the point spectrum -- we prove the spectral stability of stationary fronts in $L^2(\R)$ for localized, finite-energy perturbations.

In general, we can say that we follow the methodology of Leyva \emph{et al.} \cite{LeP20, LeLoP22}. There is, however, an important difference (an extra difficulty) when studying stationary Nagumo degenerate fronts. The stationary fronts discovered by Sanchez-Gardu\~no and Maini \cite{SaMa97} are quite peculiar. They are \emph{almost sharp} in the sense that they arrive to the degenerate state at a finite point $x = \omega_0 \in \R$ but, in contrast to sharp fronts, the connection is of class $C^1$. The latter, however, is not of class $C^2$. They appear to be smooth but they are not (see, for example, Figure \ref{figstatNagumo} below for a numerical calculation of one of such fronts). As a result, the derivative of the profile, which is the eigenfunction associated to the eigenvalue zero (or translation eigenvalue), does not belong to $H^2(\R)$; see Lemma \ref{lemC3} below. It is to be observed that all these properties of the stationary fronts were not described in the existence result in \cite{SaMa97}. Consequently, the linearized operator around the wave should be defined on a slightly larger yet more natural domain; the precise definition is given in \eqref{eq:D(L)}. This technicality makes the spectral analysis more convoluted than in \cite{LeLoP22,LeP20}, particularly with respect to energy estimates, the determination of Weyl sequences and the parabolic approximation. Despite these challenges, we prove spectral stability and extend the technique developed in \cite{LeLoP22} to control the approximate spectrum, establishing a \emph{spectral gap}, that is, a positive distance between the imaginary axis and the spectrum, excluding the origin. In addition, it is shown that the spectrum of the linearized operator lies entirely on the real line. This spectral information yields an analytic semigroup generated by the linearization around the profile and its exponential decay outside a one-dimensional eigenspace associated to translation, all crucial elements in a nonlinear asymptotic stability analysis, which will be reported in a companion paper \cite{FHLP-II}.

  
\subsection*{Main result}
Let us assume that $\varphi = \varphi(x)$ is a stationary front solution to \eqref{degRD} connecting $u = 1$ with $u = 0$. Let $\cL$ be the linearized operator around the front, a linear operator posed on the energy space $L^2(\R)$ with dense domain $\cD(\cL) \subset L^2(\R)$, for the precise definitions, see Section \ref{sec:linearized} below. 
Then, the main result of this paper can be articulated in lay terms as follows; for the precise statements see Theorem \ref{theostab} and Lemma \ref{lemexpdecay}.

\begin{stheorem}
Under assumptions \eqref{hypD} and \eqref{bistablef}, any stationary, diffusion-degenerate front solution $\varphi$ connecting $\varphi = 1$ with $\varphi = 0$ is spectrally stable in the energy space $L^2(\R)$. More precisely, $\sigma(\cL) \subset (-\infty, -\beta] \cup \{0\}$,
for a certain positive constant $\beta$. Moreover, $\lambda = 0$ is a simple eigenvalue associated to the eigenfunction $\varphi_x \in \cD(\cL)$. Finally, if $\cP$ denotes the projection onto the one-dimensional eigenspace $\mathrm{span} \{ \varphi_x \}$, then the $C_0$-semigroup generated by $\cL$, which we denote as $\{ e^{t \cL}\}_{t \geq 0}$, is analytic and satisfies $\| e^{t \cL} \cP \| \leq M e^{-\beta t}$, for all $t \geq 0$ and some $M \geq 1$.
\end{stheorem}

\subsection*{Plan of the paper} 
This work is structured as follows. In Section \ref{secstructure} we review the existence theory of stationary Nagumo degenerate fronts developed in \cite{SaMa97} and prove some new extra features, such as the existence of a finite arrival point to the degenerate state, the precise decay of the profile and its regularity. Section \ref{sec:linearized} is devoted to the definition of the linearized operator around the wave, of its domain, as well as to posing the spectral problem. The particular partition of the spectrum from Leyva \emph{et al.} \cite{LeP20,LeLoP22} is also reviewed in this Section. The central Section \ref{secspectralstab} contains the proof that the $L^2$-spectrum of the linearized operator around the stationary front is stable. Section \ref{secdecaysg} establishes the existence of an analytic semigroup generated by the linearized operator and its decaying properties. We finish the paper with some final discussion in Section \ref{secdiscuss}.

\subsection*{Notation}
We denote the real and imaginary parts of a complex number $\lambda \in \C$ by $\Re\lambda$ and $\Im\lambda$, respectively, as well as complex conjugation by ${\lambda}^*$. 
Standard Sobolev spaces of complex-valued functions on the real line will be denoted as $L^2(\R)$ and $H^m(\R)$, with $m \in \N$, endowed with the standard inner products and norms. 
In particular, in rest of the paper we use the convention
\[
\langle u, v \>_{L^2} = \int_\R u(x)^* v(x) \, dx, \qquad \qquad \| u \|^2_{L^2} = \< u, u \>_{L^2}.
\]
We use lowercase boldface roman font to indicate column vectors (e.g., $\bw$), and with the exception of the identity matrix $I$, we use upper case boldface roman font to indicate square matrices (e.g., $\bA$). Linear operators acting on infinite-dimensional spaces are indicated with calligraphic letters (e.g., $\cL$ and $\cT$), except for the identity operator which is indicated by $\Id$. We use the standard notation in asymptotic analysis \cite{Erde56}, in which the symbol ``$\sim$" means  ``behaves asymptotically like" as $x \to x_*$, more precisely, $f \sim g$ as $x \to x_*$ if $f - g = o(|g|)$ as $x \to x_*$ (or equivalently, $f/g \to 1$ as $x \to x_*$ if both functions are positive).

\section{The structure of stationary diffusion-degenerate Nagumo fronts}
\label{secstructure}

In this paper we are interested in traveling fronts solutions to equation \eqref{degRD} under hypotheses 
\eqref{hypD} and \eqref{bistablef}. The existence and fundamental properties of these solutions have been 
addressed by S\'anchez-Gardu\~no and Maini \cite{SaMa97}. For the reader's convenience, in this section we recall some of the main 
structural features of these traveling fronts and provide some additional information as well (more 
precisely, the exact decay rates of the front at equilibria and the existence of a finite arrival point to the degenerate end state).

Consider traveling wave solutions of the form $u(x,t) = \varphi(\xi)$, where $\xi=x-ct$ is the translation 
(Galilean) variable, and $c \in \R$ denotes the velocity of the front. 
Moreover, we assume that the wave profile $\varphi$ is monotone and connect the two stable points $u=0,1$.
First, we establish a relation between the integral of the product $D(u)f(u)$ and the sign of the velocity $c$.
Upon substitution, we find that the function $\varphi = \varphi(\xi)$ satisfies the equation
\begin{equation}\label{odeFront}
(D(\varphi) \varphi_\xi )_\xi + c \varphi _\xi + f(\varphi) = 0.
\end{equation}
Multiplying \eqref{odeFront} by $ D(\varphi) \varphi_\xi$, we obtain
\[
\frac{1}{2}\frac{d}{d\xi}\Big(  D(\varphi)^2 \varphi _\xi^2 \Big) + c D(\varphi)\varphi_\xi^{2}+ 
D(\varphi)  f(\varphi) \varphi_\xi=0,
\]  
and, after integrating over $\mathbb{R}$ and using that the first term vanishes, we arrive at
\[
c \int_{-\infty}^{\infty} D(\varphi) \varphi_\xi^2 \, d\xi =- \int_{-\infty}^{\infty} D(\varphi) 
f(\varphi) \varphi_\xi \, d\xi. 
\]
Observe that, because of the positive sign of the integral in the left-hand side, $c$ will always have the opposite sign of the integral on the right-hand side. 
In particular, by using the substitution $u = \varphi (\xi)$ in the integral on the left-hand side, we end up with two distinct cases depending on whether $\varphi$
is non-decreasing or non-increasing.
Indeed, one has
\begin{equation}\label{cVel}
	c =-K\int_{0}^{1} D(u) f(u) du, \qquad \mbox{ where } \qquad K:=\left(\int_{-\infty}^{\infty} D(\varphi) \varphi_\xi^2 d\xi\right)^{-1}>0,
\end{equation}
in the case $\varphi(-\infty) =0$, $\varphi(\infty)=1$ and $\varphi_\xi\geq0$ in $\mathbb R$.
On the other hand,  
\begin{equation}\label{cVel-dec}
	c =K\int_{0}^{1} D(u) f(u) du,
\end{equation}
when $\varphi(-\infty)=1 $, $\varphi(\infty)=0$ and $\varphi_\xi\leq0$ in $\mathbb R$.
Define the function $\mathscr{D}: [0,1] \to \mathbb{R}$ to be
\begin{equation}\label{defcalD}
	\mathscr{D}(\varphi) := \int_{0}^{\varphi} D(u) f(u) du.
\end{equation}
Thus, equations \eqref{cVel}-\eqref{cVel-dec} show that the sign of $c$ coincides with: 
\begin{itemize}
\item[(i)] the sign of $-\mathscr{D}(1)$ for the front connecting the states $u = 0$ and $u=1$; 
\item[(ii)] the sign  of $\mathscr{D}(1)$ for the front connecting the states $u=1$ and $u=0$.
\end{itemize}
As a consequence, if $\mathscr{D}(1)>0$ ($\mathscr{D}(1)<0)$, it is not possible to connect $u=0$ to $u=1$ ($u=1$ to $u=0$) with a non-decreasing (non-increasing) 
profile traveling with speed $c>0$ ($c<0$); see also Lemma \ref{lemNoExis}.
Conversely, if $c$ is properly chosen according to \eqref{cVel}-\eqref{cVel-dec}, then it is possible to prove the existence of the traveling front $\varphi$.
One way to prove it is to make a detailed phase 
plane analysis (for an alternate approach based on the Conley index, see \cite{ElATa10a}). 
Equation \eqref{odeFront} is written as the first order ODE system (singular at $\varphi =0$)
\begin{equation}\label{singODE}
\begin{aligned}
\frac{d \varphi}{d\xi} &= v \\
D(\varphi) \frac{dv}{d\xi} &= -cv -D'(\varphi)v^2 - f(\varphi).
\end{aligned}
\end{equation}
To overcome the singularity, Aronson \cite{Aron80} defines the change of variables,  $\tau = \tau(\xi)$, such that
\[
\frac{d\tau}{d \xi} = \frac{1}{D(\varphi(\xi))},
\]
and, therefore, system \eqref{singODE} is transformed into
\begin{equation}\label{odeSys}
\begin{aligned}
\frac{d\varphi}{d\tau} &= D(\varphi)v \\
 \frac{dv}{d\tau} &= -cv -D'(\varphi)v^2 - f(\varphi).
\end{aligned}
\end{equation}

System \eqref{odeSys} has four equilibria: $ P_0=(0,0) $, $ P_1=(1,0) $, $ P_*=(\alpha,0) $ and $ P_c 
= (0, -c/D'(0) ) $. Further analysis of the linearization of \eqref{odeSys} shows that $ P_0 $ is a 
non-hyperbolic point and $ P_1 $ is a saddle point for all values of the velocity $ c $.  Showing the existence of a heteroclinic trajectory connecting the equilibrium $ P_0 $ and $ P_1 $ accounts for the existence of a traveling front $ \varphi$ connecting the stationary states $ 
u=0 $ and $ u=1 $. S\'{a}nchez-Gardu\~{n}o and Maini \cite{SaMa97} proved  the existence of two steady wave fronts 
for equation \eqref{degRD} under hypothesis \eqref{hypD} and \eqref{bistablef}. This is summarized in the following result.

\begin{theorem}[S\'{a}nchez-Gardu\~{n}o and Maini \cite{SaMa97}]\label{exisThm}
If
\begin{equation}
\label{eq:int-Df}
\mathscr{D}(1) = \int_0^1 D(u) f(u) du = 0,
\end{equation}
then there exist two steady monotonic fronts with velocity $ c=0 $ for the equation \eqref{degRD}, one 
connecting the states $ u=0 $ and $ u=1 $ and the other connecting $u=1 $ to $ u=0 $
\end{theorem}

The authors also proved the following lemma regarding the non-existence of non-stationary traveling fronts, which is a consequence of \eqref{cVel}-\eqref{cVel-dec}. 
\begin{lemma}[\cite{SaMa97}]
\label{lemNoExis}
For $ c>0 $, system \eqref{odeSys} does not have a heteroclinic trajectory from:
\begin{itemize} 
\item $ P_0 $ to $ P_1 $, if $ \mathscr{D}(1)>0$,
\item $ P_1 $ to $ P_0 $, if $ \mathscr{D}(1)< 0$.
\end{itemize}
\end{lemma}
Similarly, one can deduce a non-existence result for the case $c<0$.
\begin{corollary}\label{cor:c<0}
For $ c<0 $ system \eqref{odeSys} does not have a heteroclinic trajectory from:
\begin{itemize} 
\item $ P_0 $ to $ P_1 $, if $ \mathscr{D}(1)<0$,
\item $ P_1 $ to $ P_0 $, if $\mathscr{D}(1) > 0$.
\end{itemize}
\end{corollary}
\begin{proof}
Again, this result can be proved thanks to \eqref{cVel}-\eqref{cVel-dec}. 
Alternatively, one can use Lemma \ref{lemNoExis} and the fact that the ODE for trajectories on the phase plane, 
\begin{equation}\label{phaseTraj}
\frac{dv}{d\varphi} = \frac{-cv-D'(\varphi) v^2 - f(\varphi)}{D(\varphi)v},
\end{equation}
is invariant under the transformation $ v\to -v $, $ c\to -c$. 
By contradiction, suppose that $\mathscr{D}(1)<0$ and that there exists a traveling front solution to equation \eqref{degRD} connecting the states $u=0$ and $u=1$ for $c<0$.
Such a solution corresponds to a heteroclinic trajectory in the phase plane connecting $P_0 $ to $P_1$. 
Hence, there exists a solution $v_1$ to equation \eqref{phaseTraj} such that $v_1(0)=0 $ and $v_1(1)=0$.
From the invariance of equation \eqref{phaseTraj} it follows that  $v=-v_1$ is a solution to \eqref{phaseTraj} for $c>0$. 
Thus, there exists a heteroclinic trajectory connecting $ P_1 $ and $ P_0 $ for $ c>0 $ when $ \mathscr{D}(1)<0$. This  contradicts Lemma \ref{lemNoExis}. Therefore, there is no traveling front connecting $ u=0 $ and $ u=1 $ for $ c<0 $, if $ \mathscr{D}(1)<0$. A similar argument can be applied to prove that there is no traveling front connecting $u=1$ to $ u=0 $ for $ c<0$, if $ \mathscr{D}(1)>0$.
\end{proof}
Combining Lemma \ref{lemNoExis} and Corollary \ref{cor:c<0}, we conclude that
\begin{lemma}\label{lem:c=0}
The only velocity $c\in\mathbb R$ for which equation \eqref{degRD} admits both a non-decreasing and a non-increasing traveling front solution 
connecting the state $ u=0 $ to $ u=1$ and $ u=1 $ to $ u=0 $, respectively, is $ c=0 $.
\end{lemma} 

From now on, we focus the attention on the case $c=0$ and since the Galilean variable $ \xi$ coincides with the spatial variable $x$, 
in the sequel we shall write $x$ instead of $ \xi $ to simplify the notation. For concreteness, and without loss of generality, for the rest of the paper we will focus our analysis on the stability of the stationary traveling front $ \varphi$ with velocity $ c 
=0 $  that satisfies
 $0 \leq \varphi (x) < 1$ for $ x \in \mathbb{R}$, and 
\begin{equation}\label{boundaryCond}
\lim_{x \to \infty} \varphi(x) = u_+ :=0, \qquad \lim_{x \to -\infty} \varphi(x) = u_- :=1,
\end{equation}
that is, it connects the equilibrium state $u_- = 1$ to the other equilibrium, the degenerate end point, $u_+ = 0$. This is called an \emph{stationary invading front}. Since $c = 0$, the profile equation \eqref{odeFront} is recast as
\begin{equation}
\label{profileeqn}
\big( D(\varphi)\varphi_x \big)_x + f(\varphi) = 0.
\end{equation}
An important feature of this stationary invading front, which it is not proved in the existence theory by S\'{a}nchez-Gardu\~{n}o and Maini \cite{SaMa97}, is the fact that it arrives at the degenerate equilibrium $u_+ = 0$ at ``finite time''.

\begin{lemma}
\label{lemfinitetime}
Under hypotheses \eqref{hypD} and \eqref{bistablef}, let us suppose that condition \eqref{eq:int-Df} holds. Then, there exists a unique value $\omega_0 \in \R$ such that the stationary front connecting $u_- = 1$ with $u_+ = 0$ satisfies $\varphi(x) \equiv 0$, for all $x \in [\omega_0, \infty)$, $\varphi_x < 0$, for all $x \in (-\infty, \omega_0)$ and $0 \leq \varphi(x) < 1$, for all $x \in \R$.
\end{lemma}
\begin{proof}
From Theorem \ref{exisThm} and Lemma \ref{lem:c=0} we know that $c=0$ is the only value of the velocity for which there exists a front connecting the state $u=0$ to $u=1$, or $u=1$ to $u=0$. Multiplying equation \eqref{profileeqn} by $D(\varphi)\varphi_x$ we obtain the following equation for the profile $\varphi$,
\[
	\left[\frac12\left(D(\varphi)\varphi_x\right)^2+\mathscr{D}(\varphi)\right]_x=0.
\]
Use \eqref{boundaryCond} and \eqref{eq:int-Df} to deduce
\begin{equation}
\label{eq:firstode}
	D(\varphi)\varphi_x=-\sqrt{-2\mathscr{D}(\varphi)},
\end{equation}
and, solving by separation of variables, we infer that the profile $\varphi$ is implicitly defined by
\[
	\int_{\varphi(0)}^{\varphi(x)}\frac{D(s)}{\sqrt{-2\mathscr{D}(s)}}\,ds=-x, \qquad \qquad \varphi(0)\in(0,1).
\]
Since we are looking for a profile $\varphi$ connecting the equilibria $0$ and $1$, it is crucial to study the convergence of the integrals
\begin{equation}
\label{integrals}
	\int_{0}^{\alpha}\frac{D(s)}{\sqrt{-2\mathscr{D}(s)}}\,ds \qquad \mbox{ and } \qquad \int_{\alpha}^1\frac{D(s)}{\sqrt{-2\mathscr{D}(s)}}\,ds.
\end{equation}
Indeed, these two integrals tell us whether the profile $\varphi$ arrives at $0$ or $1$ at some finite point 
(meaning that $\varphi(\omega_0) = 0$ or $\varphi(\omega_1)=1$ for some $\omega_0,\omega_1\in\mathbb{R}$)
or $0<\varphi(x)<1$ for any $x\in\mathbb{R}$.
The assumptions on $D$ and $f$ (conditions \eqref{hypD} and \eqref{bistablef}) allow us to determine the convergence of the integrals \eqref{integrals}:
since $D(1)>0$, $f'(1)<0$ and 
\[
\mathscr{D}(s)=\frac12D(1)f'(1)(s-1)^2+R_2(s), \quad  \mbox{ with } \quad R_2(s)=o(|s-1|^2), \;\; s\to1,
\]
where we used $\mathscr{D}(1)=\mathscr{D}'(1)=D(1)f(1)=0$, we have
\[
\int_{\alpha}^1\frac{D(s)}{\sqrt{-2\mathscr{D}(s)}}\,ds\sim\int_{\alpha}^1\frac{ds}{1-s}=\infty.
\]
This implies that the profile function $\varphi$ never ``touches'' the value $1$ and that $\varphi(x) < 1$ for all $x \in \R$ (formally speaking, that $\omega_1 = -\infty$). On the other hand, since $D'(0)>0$ and $f'(0)<0$, we have $D(s)=D'(0)s+R_1(s)$ with $R_1(s)=o(|s|)$, as $s\to0$ and 
\begin{equation}\label{expcalD}
\mathscr{D}(s)=\frac16D'(0)f'(0)s^3+R_3(s), \quad  \mbox{ with } \quad R_3(s)=o(|s|^3), \;\; s \to 0.
\end{equation}
As a consequence,
\[
\int_0^{\alpha}\frac{D(s)}{\sqrt{-2\mathscr{D}(s)}}\,ds\sim\int_0^{\alpha}\frac{ds}{\sqrt{s}}<\infty.
\]
Hence, we can conclude that assumptions \eqref{hypD} and \eqref{bistablef} imply that there exists $\omega_0\in\mathbb{R}$ such that
$\varphi(\omega_0)=0$ and $0\leq\varphi(x)<1$, for any $x\in\mathbb{R}$.
\end{proof}
\begin{example}
Consider equation \eqref{degRD} with the diffusion coefficient \eqref{Dbeta} and reaction term \eqref{cubicf}. 
The condition \eqref{eq:int-Df} implies that the parameters $ \alpha $ and $ b $  are related as follows
\[
\alpha(b)= \frac{3b+2}{5b+3}.
\]
If we set $b=1$, then $\alpha = 5/8$. 
Thus, a direct calculation shows that
\[
\mathscr{D}(\varphi) = \int_{0}^{\varphi} D(u) f(u) du = -\frac{\varphi^3}{24}(\varphi-1)^2(4\varphi+5) \leq 0,
\]
for $ \varphi \in [0,1] $. Notice that $\mathscr{D}(1) = 0$ and therefore there exists a stationary front.
Clearly $ (-\mathscr{D}(\varphi))^{1/2} $ is well defined for $\varphi \in (0,1)$. 
The integrals in \eqref{integrals} become
\begin{equation*}
	\int_{0}^{\frac58}\frac{\sqrt{12}(s+1)}{(1-s)\sqrt{s(4s+5)}}\,ds\approx3.8017, \qquad \mbox{ and } \qquad \int_{\frac58}^1\frac{\sqrt{12}(s+1)}{(1-s)\sqrt{s(4s+5)}}\,ds=\infty.
\end{equation*}
Equation \eqref{eq:firstode} becomes 
\begin{equation}
\label{newODE}
	(\varphi+1)\varphi_x=-\sqrt{\frac{\varphi}{12}(\varphi-1)^2(4\varphi+5)},
\end{equation}
and so, we have
\begin{equation*}
	\varphi_x \sim -\sqrt{\frac{5}{12}\varphi},
\end{equation*}
as $ \varphi \to 0^{+} $. 
Solving by separation of variables, we see that
\[
\varphi(x) \sim \Big(\sqrt{\varphi(0)} - \sqrt{\frac{5}{48}}x\Big)^2,
 \]
on the degenerate side, that is, when $x \to \omega_0^-$. For the other rest state we have the approximated ODE, 
\[ 
\varphi_x \sim \frac{\sqrt3}{4}(\varphi-1),
 \]
as $ \varphi \to 1^{-} $. Thus, 
\[ 
\varphi(x) \sim 1-(1-\varphi(0))e^{\frac{\sqrt{3}}{4} x},
 \]
on the non-degenerate side as $x \to 1^-$. 

Figure \ref{figstatNagumo} depicts the numerical approximation to the solution to equation \eqref{newODE}. In order to avoid the singularity of the diffusion at $\varphi = 0$, it is more convenient to solve numerically the equation
\[
\frac{dx}{d\varphi} = - \frac{\sqrt{12} (\varphi + 1)}{\sqrt{\varphi(\varphi-1)^2(4\varphi + 5)}},
\]
 with $x(\tfrac{1}{2}) = 0$ (or equivalently, when $\varphi(0) = \tfrac{1}{2}$) in the interval $\varphi \in (10^{-6}, 1- 10^{-6})$. The solution was calculated using a standard ODE integrator of implicit Runge-Kutta type. This yields the stationary monotone profile equation which is decreasing, connecting $\varphi = 1$ to $\varphi = 0$ as $x \to \infty$. From the numerical calculation it can be estimated that the ``arrival'' value for $x = \omega_0$ is approximately $\omega_0 \approx$ 2.92089.

 \begin{figure}[t]
\begin{center}
\includegraphics[scale=0.6]{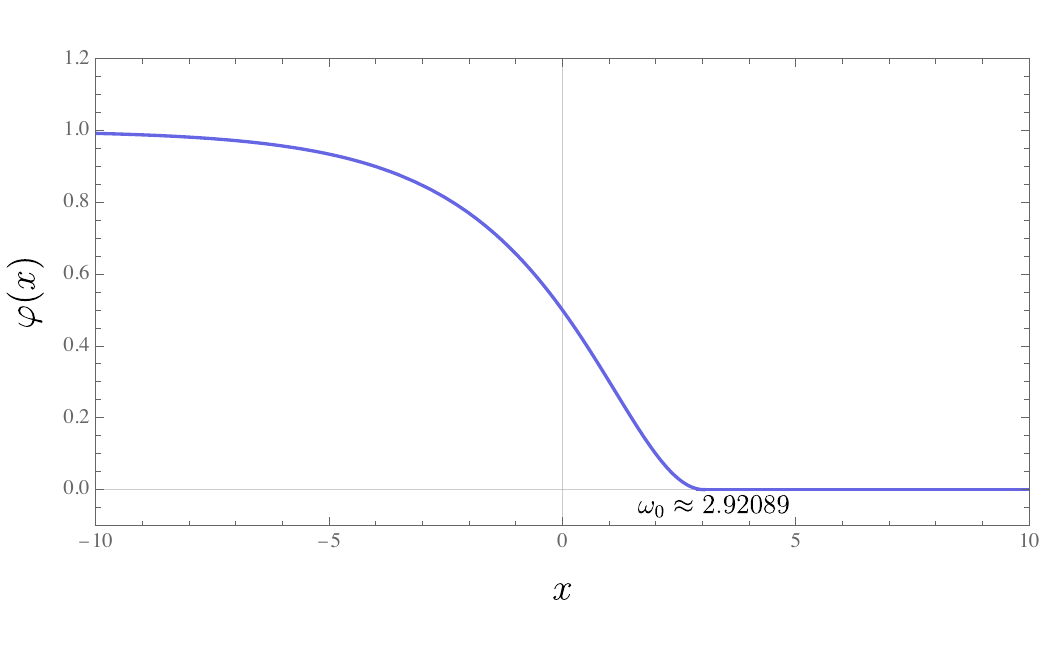}
\caption{\small{Numerical approximation of the profile of the decreasing stationary Nagumo front (in blue), connecting $\varphi = 1$ and $\varphi=0$. 
If we set $\varphi(0) = \tfrac{1}{2}$ the profile arrives to the degenerate state approximately at $x = \omega_0 \approx $ 2.92089. This connection is actually of class $C^1$ but fails to be of class $C^2$ resulting into a non-smooth profile such that $\varphi_x \notin H^2(\R)$ (color online). }
}
\label{figstatNagumo}
\end{center}
\end{figure}
 
\end{example}

This last example is very illustrative. Note that the decay of the profile to the degenerate equilibrium is algebraic, whereas it exhibits exponential decay at the non-degenerate side. As expected, the asymptotic behaviour of the front $ \varphi $ is determined by the hyperbolicity of the equilibria of system \eqref{odeSys}. Since $P_0$ is a non-hyperbolic point, the front decays to $u=0$ algebraically. On the other hand $ P_1 $ is hyperbolic and hence the front approaches $u=1$ exponentially. More precisely, we have the following result.

\begin{proposition}\label{algDecay}
The traveling front solution $\varphi$ to equation \eqref{degRD} satisfying \eqref{boundaryCond} has the following asymptotic behavior
\begin{equation}
\label{asintotasN}
\begin{aligned}
|\partial^j_x (\varphi(x) -1)| &= O(e ^{\eta x}),\quad as \; x \to -\infty, \;\; j=0,1, \\
\end{aligned}
\end{equation}
where $\eta = \sqrt{-f'(1)/D(1)} > 0$, and
\begin{equation}
\label{decaysN}
\begin{aligned}
|\partial^j_x \varphi(x)| &= O(|\omega_0 - x|^{2+j}),\quad as \; x \to \omega_0^-,  \;\; j=0,1. \\
\end{aligned}
\end{equation}
\end{proposition}
\begin{proof}
To verify the exponential decay on the non-degenerate side, note that the equilibrium point $P_1 = (1,0)$ is hyperbolic because the diffusion coefficient is not degenerate at $u=1$. The linearization of system \eqref{singODE} around $P_1$, written for $c = 0$ as 
\[
\frac{d}{dx} \begin{pmatrix} \varphi \\ v \end{pmatrix} = \mathbf{g}(\varphi, v) := \begin{pmatrix} v \\ - D(\varphi)^{-1}(D'(\varphi)v^2 + f(\varphi))\end{pmatrix},
\]
is given by
\[
(D_{(\varphi,v)}\mathbf{g})_{|(1,0)} = \begin{pmatrix}
0 & 1 \\ -f'(1)/D(1) & 0
\end{pmatrix},
\]
with a positive eigenvalue $\eta = \sqrt{-f'(1)/D(1)}$ (the decaying mode as $x \to -\infty$). By standard ODE estimates we deduce that $\varphi$ and $\varphi_x$ behave asymptotically as
\[
\varphi \sim C_0 e^{-\eta |x|}, \quad \varphi_x \sim - C_1 e^{-\eta|x|},
\]
as $x \to -\infty$ for some uniform constants $C_0, C_1 > 0$. This shows the exponential decay rate \eqref{asintotasN} at $x = -\infty$.

In order to examine the asymptotic behavior at the degenerate side, let us proceed as in the proof of Lemma \ref{lemfinitetime}.
First, from \eqref{eq:firstode}, it follows that
\[
\varphi_x= - \frac{\sqrt{- 2 \mathscr{D}(\varphi)}}{D(\varphi)},
\]
for $\varphi \in (0,1)$. 
Next, in order to find the rate of decay we make the asymptotic expansions around $\varphi = 0$ (see \eqref{expcalD}):
\begin{equation}
\label{exp2}
\begin{aligned}
D(\varphi) &= D'(0)\varphi + \tfrac{1}{2} D''(0) \varphi^2 + O(\varphi^3),\\
\mathscr{D}(\varphi) &= \tfrac{1}{3} D'(0) f'(0) \varphi^3 + O(\varphi^4).
\end{aligned}
\end{equation}
Upon substitution,
\begin{equation}
\label{exp1}
\begin{aligned}
\varphi_x &= - \frac{\sqrt{2}}{\varphi} \Big( D'(0) + \tfrac{1}{2} D''(0) \varphi + O(\varphi^2)\Big)^{-1} \Big( \tfrac{1}{3} D'(0) f'(0) \varphi^3 + O(\varphi^4)\Big)^{1/2}\\&= - a_0 \varphi^{1/2} + b_0 \varphi^{3/2} + O(\varphi^{5/2}),
\end{aligned}
\end{equation}
with $a_0 := \sqrt{- \tfrac{2}{3}f'(0)/D'(0)} > 0$ and some $b_0 > 0$. Normalizing $\varphi \to \varphi/a_0^2$ we obtain
\[
\varphi_x = - \varphi^{1/2} + b \varphi^{3/2} + O(\varphi^{5/2}),
\]
as $\varphi \to 0^+$ with $b = a_0 b_0 > 0$. By standard results of asymptotic analysis (cf. Bender and Orzag \cite{BeOr78}) we conclude that $\varphi$ behaves asymptotically like the solution $\widetilde{\varphi}$ to the equation
\[
\widetilde{\varphi}_x = - \widetilde{\varphi}^{1/2} + b \widetilde{\varphi}^{3/2},
\]
as $x \to \omega_0^-$. Integrating this equation by separation of variables one obtains
\[
\log \left( \frac{1/b + \widetilde{\varphi}^{1/2}}{1/b - \widetilde{\varphi}^{1/2}}\right) = \tfrac{1}{2}(\omega_0 -x).
\]
Solving for $\widetilde{\varphi}$ we get,
\[
\widetilde{\varphi}^{1/2} = \frac{1}{b} \frac{(e^{\tfrac{1}{2}(\omega_0 -x)}-1)}{(e^{\tfrac{1}{2}(\omega_0 -x)}+1)} \to 0^+, \quad \text{as } \; x \to \omega_0^-.
\]
Thus,
\[
\varphi \sim \widetilde{\varphi} = \frac{1}{b^2} \tanh^2 \left( \tfrac{1}{4}(\omega_0 -x)\right) = \frac{(\omega_0 - x)^2}{4b^2} + O((\omega_0 - x)^{4}),
\]
as $x \to \omega_0^-$. We then deduce the asymptotic behavior \eqref{decaysN} when $x \to \omega_0^-$ (use L'H\^opital's rule, for instance, to show the decay of the derivatives). This yields the desired algebraic decay rate on the degenerate side.
\end{proof}

We finish this section by examining the regularity of the fronts. 
Moreover, we also prove an auxiliary lemma that guarantees the boundedness of a certain coefficient, which will be used later in the paper.
\begin{lemma}[regularity]
\label{lemC3}
Under assumptions \eqref{hypD} and \eqref{bistablef}, the stationary Na\-gu\-mo front satisfies $\varphi \in C^1(\R)$ and it is of class $C^2(\R)$ except at $x = \omega_0$. Moreover, $\varphi_x \in H^1(\R)$, but $\varphi_x \notin H^2(\R)$.
\end{lemma}
\begin{proof}
Since $D \in C^2$, $f \in C^2$ and $(\varphi,v) = (\varphi, \varphi_x)$ is a solution to a nonlinear autonomous system of the form
\[
\frac{d}{dx} \begin{pmatrix} \varphi \\ \varphi_x \end{pmatrix} = \mathbf{g}(\varphi, \varphi_x),
\]
with $\mathbf{g} = \mathbf{g}(\varphi,v)$ of class $C^1$ in $(\varphi,v)$, then $\varphi$ is at least of class $C^2$ for $x \in (-\infty, \omega_0)$. Since $\varphi(x) \equiv 0$ for $x \in (\omega_0, \infty)$ we only need to verify the behavior at $x = \omega_0$. It is clear that $\varphi \in C(\R)$. From \eqref{exp1}, 
we readily see that $\lim_{x \to \omega_0^-} \varphi_x(x) = 0$, so that $\varphi \in C^1(\R)$. 
However, using the profile equation one obtains
\begin{equation*}
	\varphi_{xx} = - \frac{f(\varphi)}{D(\varphi)} - \frac{D'(\varphi) \varphi_x^2}{D(\varphi)},
\end{equation*}
for all $x \in (-\infty, \omega_0)$, where $D(\varphi) > 0$. To compute the limit when $x \to \omega_0^-$ we use the asymptotic behaviour of the profile near zero specified by \eqref{exp2} and \eqref{exp1}. For instance, substituting \eqref{exp2} we get
\[
\frac{f(\varphi)}{D(\varphi)} = \frac{f'(0) \varphi + O(\varphi^2)}{D'(0) \varphi + O(\varphi^2)} \to \frac{f'(0)}{D'(0)} <0,
\]
when $\varphi \to 0^+$ or, equivalently, when $x \to \omega_0^-$. Likewise, from \eqref{exp1} we arrive at
\[
\frac{D'(\varphi) \varphi_x^2}{D(\varphi)} = \frac{\big( D'(0) + O(\varphi) \big) \big( -a_0 \varphi^{1/2} + O(\varphi^{3/2})\big)^2}{D'(0) \varphi + O(\varphi^2)} \to a_0^2 = - \frac{2}{3} \frac{f'(0)}{D'(0)} > 0,
\]
as $\varphi \to 0^+$. Therefore we obtain
\[
\varphi_{xx} \to - \frac{1}{3} \frac{f'(0)}{D'(0)} > 0,
\]
as $x \to \omega_0^-$. Since $\varphi_{xx}$ is continuous everywhere except at $x= \omega_0$ we conclude that $\varphi \in C^2 ((-\infty,\omega_0) \cup (\omega_0, \infty))$.

The fact that $\varphi_x \in L^2(\R)$ follows immediately from being $\varphi$ of class $C^1(\R)$ for all $x \in \R$, $\varphi_x \equiv 0$ for all $x \in (\omega_0, \infty)$ and from the exponential decay of $\varphi_x$ to zero as $x \to - \infty$. However, $\varphi_{xx}\in L^2(\R)$ is discontinuous at $x = \omega_0$.
\end{proof}

\begin{remark}
\label{remnotsharp}
It is to be observed that the stationary fronts are smooth enough (at least of class $C^1$) and, as Sanchez-Gardu\~no and Maini \cite{SaMa97} point out, they are \emph{not sharp}, even though they arrive at the degenerate equilibrium point at a finite value $x = \omega_0 \in \R$. \end{remark}


\begin{lemma}
\label{lemauxi}
For the stationary Nagumo front under consideration there exists a uniform constant $C > 0$ such that
\[
\sup_{x \in \R} \left| D(\varphi) \frac{\varphi_{x x}}{\varphi_{x}}\right| \leq C.
\]
\end{lemma}
\begin{proof}
Since $\varphi_x < 0$ for all $x \in (-\infty, \omega_0)$, $\varphi = \varphi_x \equiv 0$ for all $x \geq \omega_0$ and the argument of the supremum as a  continuous function of $x \in \R$, it suffices to show that the limits,
\[
\lim_{x \to \omega_0^-} \left| D(\varphi) \frac{\varphi_{x x}}{\varphi_{x}}\right| , \quad \text{and } \quad \lim_{x \to -\infty} \left| D(\varphi) \frac{\varphi_{x x}}{\varphi_{x}}\right|,
\]
do exist. For that purpose, we use the profile equation \eqref{profileeqn} to write
\[
D(\varphi) \frac{\varphi_{x x}}{\varphi_{x}} = - \frac{f(\varphi)}{\varphi_x} - D'(\varphi) \varphi_x.
\]
Clearly, $D'(\varphi) \varphi_x \to 0$ as $x\to-\infty$ or $x \to \omega_0^-$.
On the non-degenerate side, one has 
$$\lim_{x\to-\infty}- \frac{f(\varphi)}{\varphi_x} = \sqrt{-f'(1)D(1)}>0.$$
On the other hand, from \eqref{bistablef} and \eqref{exp1}, the behavior at the degenerate side is given by 
$$- f(\varphi) / \varphi_x \sim -f'(0)\sqrt{\varphi} + O(\varphi^2) \to 0$$ as $x \to \omega_0^-$. The lemma is proved.
\end{proof}

\section{Perturbation equations and the spectral stability problem}
\label{sec:linearized}

\subsection{Preliminaries}
First, let us recall some standard concepts from the theory of linear operators (cf. \cite{Kat80,EE87}). Let $X$ and $Y$ be Banach spaces and let $\cL$ and $\cT$ be linear operators from $X$ to $Y$, with domains $\cD(\cL)$ and $\cD(\cT)$, respectively. If $\cD(\cT) \subset \cD(\cL)$ and $\cL u = \cT u$, for all $u \in \cD(\cT)$, then $\cL$ is called an extension of $\cT$ (denoted as $\cT \subset \cL$). An operator $\cT$ is bounded if $\| \cT u \| \leq M \| u \|$, for all $u \in \cD(\cT)$ and some uniform constant $M > 0$. $\cT$ is closed if, for every sequence $u_n \in \cD(\cT)$ such that $u_n \to u$ in $X$ and $\cT u_n$ is a Cauchy sequence in $Y$, then necessarily $u \in \cD(\cT)$ and $\cT u = \lim \cT u_n$, as $n \to \infty$. An operator is closable if it has a closed extension. 

Let $\ccC(X,Y)$ and $\ccB(X,Y)$ denote the sets of all closed and bounded linear operators from $X$ to 
$Y$, respectively. For any $\cL \in \ccC(X,Y)$, we denote its domain as $\cD(\cL) \subseteq X$ and its range 
as $\cR(\cL) = \cL (\cD(\cL)) \subseteq Y$. We say $\cL$ is densely defined if $\overline{\cD(\cL)} = X$. Let $\cL \in \ccC(X,Y)$ be a closed, densely defined operator;
its \textit{resolvent}, $\rho(\cL)$, is defined as the set of all complex numbers $\lambda \in \C$ such that $\cL - \lambda$ is injective, $\cR(\cL - \lambda) = Y$ and $(\cL - \lambda)^{-1}$ is bounded. The \textit{spectrum} of $\cL$ is defined as $\sigma(\cL) := \C \backslash \rho(\cL)$. 

In the analysis of stability of nonlinear waves (cf. \cite{KaPro13, San02}) the spectrum is often partitioned into essential
and isolated point spectrum. This definition is originally due to Weyl \cite{We10}, see also \cite{EE87,Kat80,KaPro13}.

\begin{definition}[Weyl's partition of spectrum]
\label{defsigmaone}
Let $\cL \in \ccC(X,Y)$ be a closed, dense\-ly defined linear operator.
We define its \textit{isolated point spectrum} and its \textit{essential spectrum}, as the sets
\[
\begin{aligned}
\iptsp(\cL) &:= \{ \lambda \in \C \, : \, \cL - \lambda \,\text{ is Fredholm with index zero and
non-trivial kernel} \}, \; \text{and}\\
\ess(\cL) &:= \{ \lambda \in \C \,: \, \cL - \lambda \text{ is either not Fredholm, or has index different 
from zero} \},
\end{aligned}
\]
respectively. 
\end{definition}

\begin{remark} 
Let us recall that an operator $\mathcal{L} \in \ccC(X,Y)$ is Fredholm if its 
range $\mathcal{R(L)}$ is closed, and both its nullity, $\nul\mathcal{L} = \dim \ker \mathcal{L}$, and 
its deficiency, $\mathrm{def} \,\mathcal{L} = \mathrm{codim} \, \mathcal{R(L)}$, are finite. $\cL$ is semi-Fredholm if $\cR(\cL)$ is closed and at least one of $\nul \cL$ and $\mathrm{def} \, \cL$ is finite. 
In both cases the index of $\cL$ is defined as
$\ind \mathcal{L} = \nul \mathcal{L} - \mathrm{def} \, \mathcal{L}$ (cf. \cite{Kat80}). Note that since $\cL$ is a closed operator, then $\sigma(\cL) = \iptsp(\cL) \cup \ess(\cL)$ (cf. Kato \cite{Kat80}, p. 167). There are many definitions of essential spectrum in the literature (see, for example, Kato \cite{Kat80} and Edmunds and Evans \cite{EE87}). Weyl's definition makes it easy to compute and has the advantage that the remaining point spectrum, $\iptsp$, is a discrete set of eigenvalues (see Remark 2.2.4 in \cite{KaPro13}). 
\end{remark}

\subsection{A particular partition of spectra}

As discussed by Leyva \emph{et al.} \cite{LeP20,LeLoP22}, in models with degenerate diffusions one may encounter technical problems when analyzing the spectrum of the linear operator $\cL$. The main obstacle is that there is loss of hyperbolicity of the coefficients when the spectral problem is written as first order system and it is not possible to apply the standard tool to locate the essential spectrum, known as Weyl's essential spectrum theorem \cite{KaPro13}. The following definition is tailored for degenerate diffusion problems.

\begin{definition}[Leyva \emph{et al.} \cite{LeLoP22,LeP20}]
 \label{defspecd}
Let $\cL \in \ccC(X,Y)$ be a closed, densely defined operator. We define the following subsets of the complex 
plane:
\[
 \begin{aligned}
 \ptsp(\cL) := &\{ \lambda \in \C \, : \, \cL - \lambda \; \text{is not injective}\},\\ 
 \spd(\cL) := &\{ \lambda \in \C \, : \, \cL - \lambda \; \text{ is injective, } \cR(\cL - \lambda) \text{ is 
closed and } \cR(\cL- \lambda) \neq Y\},\\
\sppi(\cL) := &\{ \lambda \in \C \, : \, \cL - \lambda \; \text{ is injective and } \cR(\cL - \lambda) 
\text{ is not closed}\}.
 \end{aligned}
\]
\end{definition}

\begin{remark}
Observe that the sets $\ptsp(\cL)$, $\sppi(\cL)$ and $\spd(\cL)$ are clearly disjoint and, since $\cL$ 
is closed, we know that if $\cL$ is invertible then $\cL^{-1} \in \ccB(Y,X)$. Consequently,
\[
 \sigma(\cL) = \ptsp(\cL) \cup \sppi(\cL) \cup \spd(\cL).
\] 
Also, note that the set of isolated eigenvalues with finite multiplicity (see Definition \ref{defsigmaone}) is contained in $\ptsp$, that is, $\iptsp(\cL) \subset \ptsp(\cL)$, and that $\lambda \in \ptsp(\cL)$ if and only if there exists $u \in \cD(\cL)$, $u \neq 0$ such that $\cL u = \lambda u$.
\end{remark}

\begin{remark}
\label{remponla}
For the forthcoming analysis, it is important to observe that the set $\sppi(\cL)$ is contained in the \textit{approximate spectrum} (see, e.g., \cite{EE87}), defined as
\[
\begin{aligned}
 \sppi(\cL) \subset {\sigma_\mathrm{\tiny{app}}}(\cL) := &\{ \lambda \in \C \, : \, \text{there exists $u_n \in \cD(\cL)$ with $\|u_n\| = 1$} \\ 
& \qquad\qquad \text{ such that $(\cL - \lambda)u_n \to 0$ in $Y$, as $n \to \infty$}\}.
\end{aligned}
\]
The last inclusion follows from the fact that, for any $\lambda \in \sppi(\cL)$, the range of $\cL - \lambda$ is 
not closed and, therefore, there exists a \textit{singular sequence}, $u_n \in \cD(\cL)$, $\|u_n\| =1$ such 
that $(\cL - \lambda)u_n \to 0$, which contains no convergent subsequence (see Theorems 5.10 and 5.11 in Kato \cite{Kat80}, p. 233). Moreover, if the space $X$ is reflexive then this sequence can be chosen so that $u_n$ converges weakly to 0 in $X$, denoted as $u_n \rightharpoonup 0$ in $X$ (see \cite{EE87}, p. 415). This is a key property that will be used to locate $\sppi(\cL)$ for the linearized operator around a degenerate Nagumo front. 

On the other hand, the set $\spd(\cL)$ is clearly contained in what is known as the \textit{compression spectrum} 
(see, e.g., \cite{Jerib15}, p. 196):
\[
 \spd(\cL) \subset {\sigma_\mathrm{\tiny{com}}}(\cL) := \{ \lambda \in \C \, : \, \cL - \lambda \; \text{ is 
injective, and } \overline{\cR(\cL-\lambda)} \neq Y\}.
\]
Finally, it is also clear that $\ptsp(\cL) \subset {\sigma_\mathrm{\tiny{app}}}(\cL)$. 
\end{remark}

\subsection{The linearized operator around the front}
Now, let us consider solutions to \eqref{degRD} of the form $\varphi(x)+u(x,t)$, where, from now on, $u$ denotes a 
perturbation. Substituting we obtain the nonlinear equation for perturbations,
\begin{equation}
\label{nonlinpert}
u_t = (D(\varphi+u)(\varphi+u)_x)_x  +f(u+\varphi). 
\end{equation}
Linearizing around the front and using the profile equation \eqref{profileeqn} we get
\begin{equation}\label{lineareq}
u_t = (D(\varphi)u)_{xx} +f'(\varphi)u.
\end{equation}
The right hand side of equation \eqref{lineareq} can be interpreted as a linear operator acting on an appropriate 
Banach space $X$ of perturbations. In this fashion, one naturally arrives at the following spectral problem 
\begin{equation}\label{spectralp}
	\lambda u = \mathcal{L} u,
\end{equation}
where $\lambda \in \mathbb{C}$ is the spectral parameter and
\begin{equation}
\label{opL}
\left\{
	\begin{aligned}
		\mathcal{L}&: \mathcal{D}(\mathcal{L}) \subset X \to X, \\
		\mathcal{L} u &:=(D(\varphi)u)_{xx} +f'(\varphi)u,
	\end{aligned}
\right.	
\end{equation}
is the linearized operator around the wave. Intuitively, by spectral stability we understand the absence of solutions $u \in X$ to equation \eqref{spectralp} for $\Re \lambda > 0$, precluding the existence of solutions to the linearized equation with an explosive behavior in time. Clearly, the spectrum of the operator depends upon the choice of the space of perturbations $X$. In the stability analysis of nonlinear waves (in one space dimension), it is customary to consider the energy space $X = L^2(\R)$; 
concerning the domain $\mathcal{D}(\mathcal{L})$, we consider the \emph{natural} choice 
\begin{equation}
\label{eq:D(L)}
	\mathcal{D}(\mathcal{L}):=\left\{u\in L^2(\R)\, : \, D(\varphi)u\in H^2(\R)\right\}.
\end{equation}
Notice that $H^2(\R)\subset \mathcal{D}(\mathcal{L})$ and as a consequence the linearized operator $\mathcal{L}$ is a densely defined operator acting on $L^{2}(\R)$.
Hence, the stability analysis of the operator $\mathcal{L}$ pertains to \textit{localized perturbations}. In the sequel, $\sigma(\cL)$ denotes the spectrum of any operator $\cL$ when computed with respect to the space $L^2(\R)$. This notation applies to every subset of the spectrum as well.

\begin{definition}[spectral stability]
\label{defspecstab}
 We say the traveling front $\varphi$ is \emph{spectrally stable} if 
 \[ 
 \sigma(\mathcal{L}) \subset \{ \lambda \in \C \, : \, \Re{\lambda} \leq 0\}.
  \]
 Otherwise we say that it is \emph{spectrally unstable}.
\end{definition}

It is to be observed that due to the degeneracy at $\varphi = 0$, the highest order coefficient of the operator $\mathcal{L}$ is not uniformly bounded below and, therefore, the operator $\mathcal{L}$ is, in general, not closed. It can be proved, however, that it is indeed a closed operator if we consider the particular choice of the domain \eqref{eq:D(L)}.
\begin{lemma}\label{lemclosed}
The operator $\mathcal{L} : L^2(\R) \to L^2(\R)$, with domain $\mathcal{D}(\mathcal{L})$ defined in \eqref{eq:D(L)}, is closed.
\end{lemma}
\begin{proof}
Let us prove that if the sequence $u_n\in \mathcal{D}(\mathcal{L})$ satisfies 
$$u_n\to u, \; \mbox{ in } \; L^2(\R),$$  
and $\mathcal{L}u_n$ is a Cauchy sequence in $L^2(\R)$, then automatically 
$$u\in \mathcal{D}(\mathcal{L}) \qquad \mbox{ and } \qquad \mathcal{L}u=\lim_{n\to\infty}\mathcal{L}u_n.$$
Set 
$$v:=\lim_{n\to\infty}\mathcal{L}u_n, \qquad \mbox{ in } \; L^2(\R).$$
The boundedness of the term $f'(\varphi)$ implies that
$$f'(\varphi)u_n\to f'(\varphi)u, \; \mbox{ in } \; L^2(\R).$$
Similarly, one has
$$D(\varphi)u_n\to D(\varphi)u, \; \mbox{ in } \; L^2(\R).$$
The definition \eqref{opL} of $\mathcal{L}$ gives
$$(D(\varphi)u_n)_{xx}\to v-f'(\varphi)u, \; \mbox{ in } \; L^2(\R).$$
However, $u_n\in \mathcal{D}(\mathcal{L})$ implies $D(\varphi)u_n\in H^2(\R)$ and since the operator $\partial_{xx}$ with domain $H^2(\R)$ is closed, we can conclude that
$$D(\varphi)u\in H^2(\R) \qquad \mbox{ and } \qquad (D(\varphi)u)_{xx}=v-f'(\varphi)u.$$
As a trivial consequence,
$$\mathcal{L}u=(D(\varphi)u)_{xx} +f'(\varphi)u=v.$$
In conclusion, we proved that $u\in \mathcal{D}(\mathcal{L})$ and $\mathcal{L}u=v$, that is the operator $\mathcal{L}$ with domain $\mathcal{D}(\mathcal{L})$ is closed.
\end{proof}

\begin{remark}[translation eigenvalue]\label{remef0}
In the stability theory of traveling waves, there is always a zero eigenvalue associated to translations in space of the wave. 
Indeed, in view of the profile equation \eqref{profileeqn}, it is clear that
\[ 
\mathcal{L} \varphi_x = \partial_x \big( (D(\varphi)\varphi_x)_x+f(\varphi) \big)=0.
\]
In order to conclude that $0 \in \ptsp(\cL)$, we need to verify that the derivative of the profile does belong to the domain of $\cL$.
\end{remark}

\begin{lemma}\label{lemitados}
$\varphi_x \in \cD(\cL)$ and $\cL \varphi_x = 0$.
\end{lemma}
\begin{proof}
Let us first prove that $\varphi_x \in \cD(\cL)$.
We already proved that $\varphi_x \in L^2(\R)$ (see Lemma \ref{lemC3}); from \eqref{profileeqn} it follows that
$$\big( D(\varphi)\varphi_x \big)_x= -f(\varphi).$$
Hence, $\big(D(\varphi)\varphi_x)\in H^2(\R)$ with $\big(D(\varphi)\varphi_x\big)_{xx}=-f'(\varphi)\varphi_x$ and 
$\varphi_x \in \cD(\cL)$ because of the definition \eqref{eq:D(L)}.
On the other hand, as it was previously mentioned
$$\mathcal{L} \varphi_x=(D(\varphi)\varphi_x)_{xx} +f'(\varphi)\varphi_x= \partial_x \big( (D(\varphi)\varphi_x)_x+f(\varphi) \big)=0,$$
in view of \eqref{profileeqn} and the proof is complete.
\end{proof}

\section{Spectral stability}
\label{secspectralstab}

In this section, it is proved that the $L^2$-spectrum of the linearized operator $\cL$ around the stationary degenerate Nagumo front $\varphi$ is stable. We use the techniques introduced by Leyva \emph{et al.} \cite{LeP20,LeLoP22} to handle the degeneracy of the front.

\subsection{Energy estimates and point spectral stability}
Let us perform some energy estimates in order to determine the stability of the point 
spectrum, that is, to show that $\Re \ptsp(\mathcal{L}) < 0$ except for $\lambda = 0$, which is an isolated eigenvalue 
with finite multiplicity. Recall that the zero eigenvalue is related to the invariance of a traveling wave with respect 
to translation (see Remark \ref{remef0}; in fact, $\lambda = 0 \in \ptsp(\mathcal{L})$ is associated to the eigenfunction $\varphi_x \in \cD(\cL)$). 

We start by establishing a very useful identity.

\begin{proposition}\label{lemidentity} 
Assume that $D,f$ satisfy \eqref{hypD}, \eqref{bistablef} and \eqref{eq:int-Df}.
If $\varphi$ is the stationary front solution to \eqref{degRD} studied in Section \ref{secstructure} and $\cL$ is the operator defined in \eqref{opL},
then for every $u \in H^2(\R)$ the following identity holds
\begin{equation}\label{identity}
\langle u, D(\varphi) \cL u \rangle_{L^2} = - \int_{-\infty}^{\omega_0} D(\varphi)^2  \varphi_x ^2 \left| 
\left(  \frac{u}{\varphi_x} \right)_x \right|^2 dx.
\end{equation}
\end{proposition}
\begin{proof}
If $u \in H^2(\R)$, then
\[
 \mathcal{L} u = D(\varphi) u_{xx} + 2 D(\varphi)_x u_x + (D(\varphi)_{xx} + f'(\varphi)) u.
\]
Multiply last equation by $D(\varphi)$ and rearrange terms to obtain
\begin{equation}\label{newspprblm}
D(\varphi) \mathcal{L} u = (D(\varphi)^2 u_x)_x + D(\varphi) h u,
\end{equation}
where $h:= D(\varphi)_{xx} + f'(\varphi)$. In particular, since $\mathcal{L}\varphi_x = 0$, we also have 
that
\begin{equation}\label{la2}
 (D(\varphi)^2 \varphi_{xx})_x + D(\varphi) h\varphi_x = 0, \qquad \mbox{a.e. in } x \in \R.
\end{equation}
Take the $L^2$-complex product of \eqref{newspprblm} with $u$. This yields,
\[
\langle u, D(\varphi) \mathcal{L} u \rangle_{L^2}  = \int_\R \left[u^* ( D(\varphi)^2 u_x)_x + D(\varphi) h |u|^2\right] \, dx.
\]
Notice that, since $\varphi(x) \equiv 0$ for all $x \in (\omega_0, \infty)$, this implies that
\[
\langle u, D(\varphi) \mathcal{L} u \rangle_{L^2}  = \int_{-\infty}^{\omega_0} \left[u^* ( D(\varphi)^2 u_x)_x + D(\varphi) h|u|^2\right] \, dx.
\]

For $x \in (-\infty, \omega_0)$ the front is strictly monotone, with $\varphi_x < 0$. 
Hence, we may substitute \eqref{la2}, rewritten as
\[
D(\varphi) h= - \frac{(D(\varphi)^2 \varphi_{xx} )_x}{\varphi_x},
\]
into the integral to arrive at
\[
\langle u, D(\varphi) \mathcal{L} u \rangle_{L^2}  =  \int_{-\infty}^{\omega_0} u^* ( D(\varphi)^2 u_x)_x \, dx - \int_{-\infty}^{\omega_0} \frac{(D(\varphi)^2 \varphi_{xx} )_x}{\varphi_x} |u|^2 \, dx. 
\]
Integrating by parts we obtain,
\begin{align*}
	\int_{-\infty}^{\omega_0} u^* ( D(\varphi)^2 u_x)_x \, dx &= - \int_{-\infty}^{\omega_0} D(\varphi)^2 |u_x|^2 \, dx + \big( u^* D(\varphi)^2 u_x \big) \Big|_{-\infty}^{\omega_0} \\
	&= - \int_{-\infty}^{\omega_0} D(\varphi)^2 |u_x|^2 \, dx,
\end{align*}
because $D(\varphi) \to D(0) = 0$ as $x \to \omega_0^-$ and $|u|,|u_x| \to 0$ as $x \to - \infty$. On the other hand, 
\[
\int_{-\infty}^{\omega_0} \frac{(D(\varphi)^2 \varphi_{xx} )_x}{\varphi_x} |u|^2 \, dx = - \int_{-\infty}^{\omega_0} \left( \frac{|u|^2}{\varphi_x}\right)_x D(\varphi)^2 \varphi_{xx} \, dx + \left. \Big( \frac{|u|^2 D(\varphi)^2 \varphi_{xx}}{\varphi_x}\Big) \right|_{-\infty}^{\omega_0}.
\]
But by Lemma \ref{lemauxi}, there exists a uniform constant $C > 0$ such that
\[
0 \leq |u|^2 \Big| \frac{D(\varphi)^2 \varphi_{xx}}{\varphi_x} \Big| \leq C |u|^2 D(\varphi) \to 0,
\]
as $x \to - \infty$ and as $x \to \omega_0^-$. Therefore, the boundary terms vanish and we arrive at
\begin{equation*}
\langle u, D(\varphi) \mathcal{L} u \rangle_{L^2}  = \int_{-\infty}^{\omega_0} D(\varphi)^2 \left( \varphi_{xx} \left(  
\frac{|u|^2}{\varphi_x} \right)_x  - |u_x| ^2\right)dx .
\end{equation*}
Now, using the identity
\[ \varphi_x ^2 \left| \left(  \frac{u}{\varphi_x} \right)_x \right|^2 = - \left( \varphi_{xx} 
\left(  \frac{|u|^2}{\varphi_x} \right)_x  - |u_x| ^2\right), \qquad x \in (-\infty, \omega_0),
\]
we get 
\begin{equation*}
\langle u, D(\varphi) \mathcal{L} u \rangle_{L^2}  = - \int_{-\infty}^{\omega_0} D(\varphi)^2  \varphi_x ^2 \left| 
\left(  \frac{u}{\varphi_x} \right)_x \right|^2 dx,
\end{equation*}
for every $u \in H^2(\R)$.
\end{proof}
The next goal is to extend the identity \eqref{identity} to the whole domain $\cD(\cL)$. To do that, we prove the following result.
\begin{lemma}\label{lem:core}
Assume that $D,f$ satisfy \eqref{hypD}, \eqref{bistablef} and \eqref{eq:int-Df}.
If $\varphi$ is the traveling front solutions to \eqref{degRD} studied in Section \ref{secstructure} and $\cL$ is the operator defined in \eqref{opL},
then for any $u\in\cD(\cL)$, there exists a sequence $u_n\in H^2(\R)$ such that
\begin{equation}
\label{eq:core}
	u_n\to u, \quad \mbox{ in } L^2(\R) \qquad \mbox{ and } \qquad \cL u_n\to\cL u, \quad \mbox{ in } L^2(\R).
\end{equation}
\end{lemma}
\begin{proof}
Fix $u\in\cD(\cL)$ and define $w:=D(\varphi)u\in H^2(\R)$.
In particular, notice that $w\in C^1(\R)$ and, since $D(\varphi)=0$, for any $x\geq\omega_0$ and $u\in L^2(\R)$, we deduce that
$$w(x)=0, \quad \mbox{ for any } x\geq\omega_0, \qquad \qquad w(\omega_0)=w'(\omega_0)=0.$$
Let us define a smooth cutoff function $\chi\in C^\infty(\R)$ such that
$$\chi(x)=0, \quad \mbox{ for } x\geq0, \qquad \mbox{ and } \qquad \chi(x)=1, \quad \mbox{ for } x\leq-1.$$
Set $\chi_n(x):=\chi(n(x-\omega_0)+1)$, then $\chi_n\in C^\infty(\R)$ satisfies $|\chi'_n|\leq Cn$, for any $x\in\R$ and
$$\chi_n(x)=0, \quad \mbox{ for } x\geq\omega_0-1/n, \qquad \mbox{ and } \qquad \chi_n(x)=1, \quad \mbox{ for } x\leq\omega_0-2/n.$$
It is easy to check that the sequence $w_n=\chi_n w \in H^2(\R)$ satisfies 
\begin{equation}\label{eq:crucial}
	w_n(x)=0, \qquad \qquad \mbox{ if } x\geq\omega_0-\frac1n, \,\mbox{ for any }n\in\N.
\end{equation} 
Moreover, since $|w(x)|\leq C|\omega_0-x|^2$, for $x$ near $\omega_0$, we obtain 
$$\lim_{n\to\infty}w_n= w, \qquad \mbox{ in } H^2(\R).$$
Considering that $D(\varphi(\omega_0))=0$, we can not state that $w_n/D(\varphi)\in H^2(\R)$;
thus, we introduce the sequence
$$\tilde w_n(x):=\left(w_n\,*\,\rho_{1/n}\right)(x)=\int_{-\infty}^{\infty}w_n(y)\rho_{1/n}(x-y)\,dy=\int_{-\infty}^{\infty}w_n(x-y)\rho_{1/n}(y)\,dy,$$
where $\rho_{1/n}(x):=n\rho(nx)$, with $\rho\in C^\infty_c(0,1)$, satisfying $\displaystyle\int_0^1 \rho=1$.
It is well known that $\tilde w_n(x)=0$, for $x\geq\omega_0$, $\tilde w_n\in C^\infty(\R)$ and
$$\tilde w_n^{(k)}(x)=\left(w_n\,*\,\rho_{1/n}^{(k)}\right)(x)=\int_{0}^{1/n}w_n(x-y)\rho_{1/n}^{(k)}(y)\,dy,$$
for any $k\in\N$.
In particular, thanks to the crucial property \eqref{eq:crucial}, we end up with
$$\tilde w_n^{(k)}(\omega_0)=\int_{0}^{1/n}w_n(\omega_0-y)\rho_{1/n}^{(k)}(y)\,dy=0,$$  
for any $k\in\N$.
Finally, one has
$$\lim_{n\to+\infty}\tilde w_n=w=D(\varphi)u, \qquad \mbox{ in } H^2(\R),$$
since 
$$\|\tilde w_n-w\|_{H^2}\leq\|\tilde w_n-w_n\|_{H^2}+\|w_n-w\|_{H^2}\to 0, \qquad \mbox{ as } n\to\infty.$$ 
Now, the function
$$u_n:=
\begin{cases}
\displaystyle\frac{\tilde w_n}{D(\varphi)}, \qquad & x<\omega_0,\\
0, & x\geq\omega_0,
\end{cases}
$$
satisfies $u_n\in H^2(\R)$, $u_n\to u$, in $L^2(\R)$, $D(\varphi)u_n\to D(\varphi)u$, in $H^2(\R)$ and, as a consequence, we proved \eqref{eq:core}.
\end{proof}

Thanks to Lemma \ref{lem:core}, we can extend the identity \eqref{identity} to the whole domain $\cD(\cL)$.
\begin{proposition}
Under the same assumption of Proposition \ref{lemidentity}, the identity \eqref{identity} holds true for any $u\in\cD(\cL)$.
\end{proposition}
\begin{proof}
To begin with, we prove that the right-hand side of \eqref{identity} is well defined for $u\in\cD(\cL)$.
Indeed, one can rewrite 
\begin{equation}\label{eq:trick}
	D(\varphi)\varphi_x\left( \frac{u}{\varphi_x}\right)_x=D(\varphi)\varphi_x\left( \frac{D(\varphi)u}{D(\varphi)\varphi_x}\right)_x
	=\left[D(\varphi)u\right]_x-\frac{D(\varphi)\varphi_{xx}}{\varphi_x}u-D'(\varphi)\varphi_xu,
\end{equation}
and the definition of the domain \eqref{eq:D(L)} together with Lemma \ref{lemauxi} imply that 
$$\mbox{if } u\in\cD(\cL), \qquad \mbox{ then } \qquad D(\varphi)\varphi_x\left( \frac{u}{\varphi_x}\right)_x\in L^2(\R).$$

Next, fix $u\in\cD(\cL)$ and consider $u_n\in H^2(\R)$ such that \eqref{eq:core} holds true. 
Thus, Proposition \ref{lemidentity} implies 
\begin{equation}\label{eq:density}
	\langle u_n, D(\varphi) \cL u_n\rangle_{L^2} = - \int_{-\infty}^{\omega_0} D(\varphi)^2  \varphi_x ^2 \left|\left(\frac{u_n}{\varphi_x} \right)_x \right|^2 dx,
\end{equation}
for any $n\in\mathbb N$.
Regarding the left-hand side, one has
\begin{align*}
	\left|\langle u_n, D(\varphi) \cL u_n\rangle_{L^2}-\langle u, D(\varphi) \cL u\rangle_{L^2} \right|&\leq \left|\langle u_n-u, D(\varphi) \cL u_n\rangle_{L^2}\right|\\ 
	& \qquad +\left|\langle u, D(\varphi) \left(\cL u_n-\cL u\right)\rangle_{L^2}\right|\\
	&\leq\|D(\varphi)\|_{L^\infty}\|u_n-u\|_{L^2}\|\cL u_n\|_{L^2}\\
	&\quad + \|D(\varphi)\|_{L^\infty}\|u\|_{L^2}\|\cL u_n-\cL u\|_{L^2},
\end{align*}
and from \eqref{eq:core} it follows that
\begin{equation}\label{eq:part1}
	\lim_{n\to\infty}\langle u_n, D(\varphi) \cL u_n\rangle_{L^2}=\langle u, D(\varphi) \cL u\rangle_{L^2}.
\end{equation}
On the other hand, by using \eqref{eq:trick} we deduce
\begin{align*}
	\left\|D(\varphi) \varphi_x\left(\frac{u_n}{\varphi_x} \right)_x-D(\varphi) \varphi_x\left(\frac{u}{\varphi_x} \right)_x\right\|_{L^2}
	&\leq\left\|\left[D(\varphi)u_n\right]_x-\left[D(\varphi)u\right]_x\right\|_{L^2}\\
	&\qquad +\left\|\frac{D(\varphi)\varphi_{xx}}{\varphi_x}(u-u_n)\right\|_{L^2}\\
	&\qquad +\left\|D'(\varphi)\varphi_x(u-u_n)\right\|_{L^2}.
\end{align*}
Again, by using \eqref{eq:core} and Lemma \ref{lemauxi}, we end up with
\begin{equation}\label{eq:part2}
	\lim_{n\to\infty}D(\varphi) \varphi_x\left(\frac{u_n}{\varphi_x} \right)_x= D(\varphi) \varphi_x\left(\frac{u}{\varphi_x} \right)_x, \qquad \mbox{ in } L^2(\R).
\end{equation}
Thanks to \eqref{eq:part1} and \eqref{eq:part2}, passing to the limit as $n\to\infty$ in \eqref{eq:density}, we conclude that \eqref{identity}
holds true for any $u\in\cD(\cL)$ and the proof is complete.
\end{proof}

Now, we have all the tools to prove the stability of the point spectrum $\ptsp(\mathcal{L})$.
\begin{theorem}
\label{lempstab} 
Assume that $D,f$ satisfy \eqref{hypD}, \eqref{bistablef} and \eqref{eq:int-Df}. If $\varphi$ is the traveling front solutions to \eqref{degRD} studied in Section \ref{secstructure} and $\cL$ is the operator defined in \eqref{opL}, then
$\ptsp(\mathcal{L}) \subset (-\infty,0]$.
\end{theorem}
\begin{proof}
Suppose that $\lambda \in \ptsp(\mathcal{L}) \subset \C$. Let $u \in \cD(\cL)$ be a solution to
\[
 \mathcal{L} u = \lambda u.
\]
Multiply last equation by $D(\varphi)$ and take the $L^2$-complex product with $u$ to obtain
\begin{equation*}
 	\langle u, D(\varphi) \cL u \rangle_{L^2} = \langle u, D(\varphi) \lambda u \rangle_{L^2}.
\end{equation*}
By using \eqref{identity} and rearranging the terms we end up with
\begin{equation}
\label{stableEnerStim}
\lambda \int_{-\infty}^{\omega_0} D(\varphi) |u|^2 dx = - \int_{-\infty}^{\omega_0} D(\varphi)^2  \varphi_x ^2 \left| 
\left(  \frac{u}{\varphi_x} \right)_x \right|^2 dx.
\end{equation}
Since $D(\varphi) \geq 0$, this shows that $\lambda$ is real with $\lambda \leq 0$. The theorem is proved.
\end{proof}

\begin{corollary}
\label{corsingleev}
$\ker \mathcal{L} = \mathrm{span} \, \{ \varphi_x\}$, that is, the geometric multiplicity of $\lambda = 0$ is 
equal to 
one.
\end{corollary}
\begin{proof}
Let $u \in \ker \mathcal{L}$. From the identity \eqref{stableEnerStim} and since $\varphi_x < 0$ and $D(\varphi) > 0$ for $x \in (-\infty, \omega_0)$, we deduce that $\lambda = 0$ if and only if
 \[
  \left( \frac{u}{\varphi_x}\right)_x = 0, \quad \text{a.e. in } \, x \in (-\infty, \omega_0),
 \]
yielding $u = \kappa_1 \varphi_x +\kappa_2  $ a.e. in $x \in (-\infty, \omega_0)$ for some constants $\kappa_j \in \C$. Taking the limit when $x \to - \infty$ we obtain $\kappa_2 = 0$ and hence $u = \kappa_1 \varphi_x$ a.e. in $x \in (-\infty, \omega_0)$. 

For $x \in (\omega_0, \infty)$ we know that $\varphi(x) = 0$ and the differential equation becomes $0 = \cL u = - f'(0) u$ because $D(0) = 0$. Since $f'(0) < 0$ we conclude that $u = 0 = \varphi_x$ a.e. in  $x \in (\omega_0,\infty)$ as well and this shows that $u \in \mathrm{span} \, \{ \varphi_x\}$. The corollary is proved.
\end{proof}

Let us now examine the algebraic multiplicity of the eigenvalue $\lambda = 0$. To that end, assume $u\in H^2(\R)$ and rewrite the operator $\mathcal{L}$ as
\[ 
\mathcal{L} u  = a_2(x) u_{xx} + a_1(x) u_x + a_0(x) u ,
\]
 where $a_2(x) = D(\varphi)$, $a_1(x) = 2 D(\varphi)_x$, and $a_0(x)= D(\varphi)_{xx} + f'(\varphi)$.
Consider the transformation
\[
w = z^{-1} u, \qquad z:=z(x) = \exp\left ( \int_{x_0}^{x} b(y) \, dy \right),
\]
with $b:=b(x)$ and $x_0 \in (-\infty, \omega_0)$ to be chosen later. Upon substitution,
\[ 
\mathcal{L}u = \mathcal{L}(wz) = z \big( a_2 w_{xx} + (2a_2 b +a_1) w_x + (a_2 (b_x +b^2 )+ a_1 b + 
a_0)w \big).
 \]
Choosing $ b(x) = - a_1(x) / 2a_2(x) $, the coefficient of $w_x$ vanishes and we obtain
\[ 
\mathcal{L}u = z (a_2 w_{xx} + c(x) w),
\]
where 
\[
c(x)= - \frac{1}{4} \frac{a_1(x)^2}{a_2(x)} - \frac{1}{2} \partial_x a_1(x)+ \frac{1}{2}\frac{a_1(x) 
\partial _x a_2(x)}{a_2(x)} + a_0(x) = f'(\varphi),
\]
after substituting the expressions of $a_2, a_1, a_0$ and their derivatives. Therefore, we arrive at the 
relation
\begin{equation}
\label{relationLadj}
 \mathcal{L} u = \mathcal{L}(zw) = z ( D(\varphi) w_{xx} + f'(\varphi)w ) = z \mathcal{L}^* w,
\end{equation}
where $\mathcal{L}^* : L^2(\R) \to L^2(\R)$ is the formal adjoint of the operator $\mathcal{L}$. 
We calculate $z$ explicitly; fix any $x_0 \in (-\infty, \omega_0)$ and let
\begin{align*}
	 z(x) & = \exp \left( - \int_0^x \frac{a_1(y)}{2 a_2(y)} \, dy\right) = \exp \left( - \int_0^x \frac{D(\varphi)_y}{D(\varphi)} \, dy\right) \\
	 &= \frac{D_0}{D(\varphi)}, \qquad \qquad x \in (-\infty,\omega_0),
\end{align*}
where $D_0 := D(\varphi(x_0)) > 0$. 
Therefore, if $u \in \cD(\cL)$ then $w = D_0^{-1}D(\varphi) u \in H^2(\R)$ because of the definition of $\cD(\cL)$, see \eqref{eq:D(L)}.

It is known \cite{Kat80} that if $\lambda \in \ptsp(\mathcal{L})$, $\mathcal{L}$ closed, then $\lambda^* \in 
\ptsp(\mathcal{L}^*)$ with same geometric and algebraic multiplicities. Thus, $0 \in \ptsp(\mathcal{L^*})$, 
associated to some eigenfunction $\psi \in H^2(\R)$. Equation \eqref{relationLadj}, however, establishes 
an isomorphism between the kernels of $\mathcal{L}$ and $\mathcal{L}^*$ and provides an explicit expression 
for $\psi$:
\begin{equation}
 \label{exprpsi}
 \begin{aligned}
 \psi(x) &= z(x)^{-1} \varphi_x = \frac{D(\varphi)}{D_0} \varphi_x \, \in \, H^2(\R),\\
 \mathcal{L}^* \psi &= z(x)^{-1} \mathcal{L} \varphi_x = 0.
 \end{aligned}
\end{equation}
At this point we define,
\begin{equation}
\label{defTheta}
\Theta := \langle \psi, \varphi_x \rangle_{L^2} = \langle (D(\varphi)/D_0) \varphi_x, \varphi_x 
\rangle_{L^2} = \int_\R \frac{D(\varphi)}{D_0} |\varphi_x|^2 \, dx > 0.
\end{equation}
 
Thus, we have proved the following
\begin{lemma}
 \label{lempre1}
$\lambda = 0$ is an eigenvalue of the formal adjoint $\cL^* : H^2(\R) \to 
L^2(\R)$, and there exists an eigenfunction $\psi \in H^2(\R)$ such that $\ker 
(\mathcal{L}^*) = \mathrm{span} \, \{ \psi \}$. Moreover, $\Theta = \langle \psi, \varphi_x \rangle_{L^2} > 
0$.
\end{lemma}

Finally, we show that $\lambda = 0$ is an isolated simple eigenvalue. 
\begin{lemma}\label{lemzerosimple}
 $\lambda = 0$ is an isolated eigenvalue of $\mathcal{L}$ with algebraic multiplicity equal to one. 
\end{lemma}
\begin{proof}
Consider a Jordan chain for $\lambda = 0$: suppose there exists $u_1 \in \cD(\cL)$ such that $\mathcal{L} 
u_1 = \varphi_x$. Then, equation \eqref{relationLadj} implies that
\begin{equation}\label{jChainLadj}
\mathcal{L}^* w_1 = z^{-1}\varphi_x = \frac{D(\varphi)}{D_0} \varphi_x,
\end{equation}
where $w_1 = z^{-1 }u_1 \in H^2(\R)$. By virtue of Fredholm's alternative, equation \eqref{jChainLadj} has a 
solution $w_1 \in H^2(\R)$ if and only if $z^{-1} \varphi_x \in (\ker \mathcal{L})^{\perp}$, that is,
\[
 \langle z^{-1} \varphi_x, \varphi_x \rangle_{L^2} = \langle \psi, \varphi_x \rangle_{L^2} = \Theta = 0,
\]
which is a contradiction, since \eqref{defTheta} holds true. This shows that the Jordan chain for the operator $\mathcal{L}$ has length 
equal to one, and the algebraic multiplicity of $\lambda =0$ is equal to one as well.
\end{proof}

To conclude this subsection, we observe that, in view of Theorem \ref{lempstab}, Corollary \ref{corsingleev} and Lemma \ref{lemzerosimple}, there is a \emph{spectral gap} for the point spectrum, 
namely, there exists $\mu_0>0$ such that
\[
\ptsp(\mathcal{L}) \subset (-\infty, - \mu_0] \cup \{ 0\}.
\]
This holds because the point spectrum comprises discrete eigenvalues and for any $\lambda \in \ptsp(\mathcal{L})$ with $\lambda \neq 0$, there is an eigenfunction $u \notin \text{span} \{ \varphi_x \}$, $u \neq 0$, for which identity \eqref{stableEnerStim} readily implies that $\lambda \in \R$, $\lambda < 0$. Hence, it suffices to take $\mu_0$ smaller than the modulus of the first strictly negative eigenvalue. Hence, we can refine the statement of Theorem \ref{lempstab} into the following result.

\begin{theorem}[point spectral stability]
\label{thmpstab} 
Under assumptions \eqref{hypD}, \eqref{bistablef} and \eqref{eq:int-Df}, the point spectrum of the linearized operator around the stationary diffusion degenerate Nagumo front is real and stable. More precisely, there exists $\mu_0 > 0$ such that
\begin{equation}
\label{ptspstabmu0}
\ptsp(\mathcal{L}) \subset (-\infty, - \mu_0] \cup \{ 0\}.
\end{equation}
Moreover, $\lambda = 0$ is a simple eigenvalue with associated eigenfunction $\varphi_x \in \cD(\cL)$.
\end{theorem}

\subsection{Parabolic regularization and stability of $\sigma_\delta$}

In this section, we apply the regularization technique developed by Leyva \emph{et al.} \cite{LeP20,LeLoP22} in order to determine the stability of a subset of the essential spectrum. Let $\varphi = \varphi(x)$ be the stationary degenerate Nagumo front and consider the following \emph{regularization} of the diffusion coefficient: for any $\epsilon > 0$, let us define
\begin{equation}
\label{regD}
D^\epsilon(\varphi) := D(\varphi) + \epsilon.
\end{equation}
Note that, for fixed $\epsilon > 0$, the coefficient $D^\epsilon(\varphi)$ is positive and uniformly bounded below, including the asymptotic limits, $D^\epsilon(0) = \epsilon > 0$ and $D^\epsilon(1) = D(1) + \epsilon > 0$. This allows us to invoke the standard Weyl's essential spectrum theorem (see \cite{KaPro13}) for the following family of linearized, regularized operators,
\begin{equation}
\label{perturbOpL}
\left\{
\begin{aligned}
&\cL^\epsilon : L^2(\R) \to L^2(\R),\\
&\cD(\cL^\epsilon) = H^2(\R),\\
&\cL^\epsilon u := (D^\epsilon(\varphi) u)_{xx} + f'(\varphi)u, \quad u \in \cD(\cL^\epsilon),
\end{aligned}
\right.
\end{equation}
for any $\epsilon > 0$. Here, $\cL^\epsilon$ is a densely defined, closed operator in $L^2(\R)$ for any $\epsilon>0$ (see Proposition 3.10 and Section 3.3 in \cite{LeLoP22}). Actually, as it was proved in Section \ref{sec:linearized}, also for $\epsilon = 0$, $\cL^0 = \cL$ is a densely defined, closed operator in $L^2(\R)$. Notice that $\cD(\cL^\epsilon) = H^2(\R)\subset\cD(\cL)$, for any $\epsilon>0$.

At this point, let us define the region of \emph{consistent splitting} (cf. \cite{AGJ90,San02,KaPro13})
\[
\Omega:= \left\{ \lambda \in \C \, : \, \Re \lambda > \max \{ f'(0), \, f'(1) \} \right\},
\]
with $\Omega$ that does not depend on $\epsilon$. In this region, the coefficients of the spectral first order system (when the spectral equation for $\cL^\epsilon$ is recast as a dynamical system) have no center eigenspace and the dimensions of the stable/unstable manifolds coincide (they are consistent); see Section 4 in \cite{LeLoP22} or Chapter 3 in \cite{KaPro13}. Hence it is very natural to relate the Fredholm properties of the operator with those of its asymptotic counterpart. Hence, standard methods immediately imply the following result.
\begin{lemma}
\label{lemLep}
For any $\epsilon > 0$ and for each $\lambda \in \Omega$, the operator $\cL^\epsilon - \lambda$ is Fredholm with zero index.
\end{lemma}
\begin{proof}
See the proof of Lemma 4.2 in \cite{LeLoP22} with $a = 0$.
\end{proof}

However, since we are studying a stationary front with zero speed, $c = 0$, and its linearization in an unweighted Sobolev energy space (for which $a = 0$ in the nomeclature of \cite{LeLoP22}), we are able to provide more information. Let us take a look at the aforementioned equivalent first order system. Recast the spectral problem $(\cL^\ep - \lambda)u = 0$, $u \in D(\cL^\ep) = H^2(\R)$, as
\begin{equation}
\label{fosyst}
\bw_x = \bA(x,\lambda) \bw,
\end{equation}
where
\[
\bw = \begin{pmatrix} u \\ u_x \end{pmatrix} \in H^2(\R) \times H^1(\R), \quad \bA(x,\lambda) = \begin{pmatrix} 0 & 1 \\ D^\ep(\varphi)^{-1}(\lambda - b_0^\ep(x)) & - D^\ep(\varphi)^{-1} b_1^\ep(x) \end{pmatrix},
\]
\[
b_1^\ep(x) := 2 D^\ep(\varphi)_x, \quad b_0^\ep(x) := D^\ep(\varphi)_{xx} + f'(\varphi),
\]
(see, e.g., \cite{LeLoP22}). Setting
\[
\bA_\pm^\ep(\lambda) := \lim_{x \to \pm \infty } \bA(x,\lambda),
\]
we obtain
\[
\bA_\pm^\ep(\lambda) = \begin{pmatrix} 0 & 1 \\ D^\ep(u_\pm)^{-1}(\lambda - f'(u_\pm)) & 0 \end{pmatrix},
\]
with $u_+ = 0$, $u_- = 1$. The boundary of the essential spectrum of $\cL^\ep$ is determined by the following dispersion relation (see \cite{KaPro13,San02,LeLoP22}),
\[
\pi^\ep_{\pm}(k,\lambda) = \det (\bA_\pm^\ep(\lambda) - ik I) = 0, \qquad k \in \R,
\]
which determines the complex values of $\lambda$ for which the matrices $\bA_\pm^\ep(\lambda)$ fail to be hyperbolic. This dispersion relation defines two curves in the complex plane (known as Fredholm borders),
\[
\lambda_{\pm}^\ep(k) = - D^\ep(u_\pm) k^2 + f'(u_\pm), \qquad k \in \R.
\]
Since $f'(1), f'(0) < 0$, we notice that both curves lie in the negative real axis, 
\[
\begin{aligned}
\lambda_{+}^\ep(k) &= - D^\ep(0) k^2 + f'(0) = - \ep k^2 + f'(0) < 0,\\
\lambda_{-}^\ep(k) &= - D^\ep(1) k^2 + f'(0) = - (\ep + D(1)) k^2 + f'(1) < 0.
\end{aligned}
\]
Consequently we have the following refinement of Lemma \ref{lemLep}.
\begin{lemma}
\label{lemLep2}
Define
\[
\mu_1 := \min \{ |f'(1)|, |f'(0)|\} > 0.
\]
Then for each $\lambda \in \C \backslash (-\infty, \mu_1]$ and every $\ep > 0$, the operator $\cL^\epsilon - \lambda$ is Fredholm with zero index.
\end{lemma}
\begin{proof}
Fix $\lambda \in \C \backslash (-\infty, \mu_1]$ and $\ep > 0$. Then the matrices $\bA_\pm^\ep(\lambda)$ are clearly hyperbolic. Hence, from exponential dichotomies theory (see Coppel \cite{Cop78}, or Theorem 3.3 in \cite{San02}), the system \eqref{fosyst} is endowed with exponential dichotomies on both rays, $(-\infty,0)$ and $(0,\infty)$, with Morse indices $i^+ = \dim U_+^\ep(\lambda) = 1$ and $i^-(\lambda) = \dim S_{-}^\ep(\lambda) = 1$, respectively (see Lemma 4.1 in \cite{LeLoP22} in the particular case $a = 0$ and $c = 0$; here $U_\pm^\ep(\lambda)$ and $S_\pm^\ep(\lambda)$ denote the unstable and stable spaces of $\bA_\pm^\ep(\lambda)$, respectively). Therefore, Theorem 3.2 in Sandstede \cite{San02} implies that the operators $\partial_x - \bA^\ep(x,\lambda)$ (and hence the operators $\cL^\ep - \lambda$) are Fredholm with index equal to $i^+ - i^- = 0$. This proves the result.
\end{proof}

This last Lemma readily implies the stability of the essential spectrum of each $\cL^\epsilon$ (directly from the definition of $\ess$). In the limit when $\epsilon \to 0$, however, we apply Kato's stability theorem (see \cite{Kat80}, p. 235) which demands the range of the operator to be closed. Therefore, the definition of $\sigma_\delta$ is designed to sort out the values of $\lambda$ for which the range of $\cL- \lambda$ is closed. As a result, we are not able to control the whole essential spectrum of $\cL$ with this regularization technique, but only a subset. The following result is the main tool to locate the subset of the compression spectrum, $\sigma_\delta(\cL)$, in the degenerate front case. 
\begin{lemma}
\label{lemLsemi}
Suppose that $\cL - \lambda$ is semi-Fredholm for some $\lambda \in \C$. Then for each $0 < \epsilon \ll 1$ sufficiently small $\cL^\epsilon - \lambda$ is semi-Fredholm and $\ind(\cL^\epsilon - \lambda) = \ind(\cL - \lambda)$.
\end{lemma}
\begin{proof}
This is a particular case (with $a = 0$) of Lemma 4.7 in \cite{LeLoP22}, with the slight caveat that $\cD(\cL^\ep) = H^2(\R) \neq \cD(\cL)$. We know, however, that $H^2(\R) $ is densely embedded into $\cD(\cL)$ and the infima of the difference norm over the unit ball on the graph of $\cL - \lambda$  can be taken on a subset of $H^2 \times L^2$ (details are left the reader). The rest of the proof follows that of Lemma 4.7 in \cite{LeLoP22}, word by word.
\end{proof}

With this information at hand, let us now locate the subset $\sigma_\delta(\cL)$ of the compression spectrum.

\begin{theorem}\label{thmstabsdelta}
The $\sigma_\delta$-spectrum of $\cL$ is stable. More precisely, 
\[
\sigma_\delta(\cL) \subset (-\infty, -\mu_1].
\]
\end{theorem}
\begin{proof}
First, it is to be observed that
\[
\sigma_\delta(\cL) \subset \{ \lambda \in \C \, : \, \cL - \lambda \; \text{is semi-Fredholm with ind}  (\cL - \lambda) \neq 0 \}. 
\]
Indeed, if $\lambda \in \sigma_\delta(\cL)$ then from Definition \ref{defspecd} we know that $\cL - \lambda : L^2(\R) \to L^2(\R)$ is injective, $\cR(\cL - \lambda)$ is closed in $L^2(\R)$ and $\cR(\cL - \lambda) \subsetneqq L^2(\R)$. This implies that $\nul (\cL - \lambda) = 0$ and consequently $\cL - \lambda$ is semi-Fredholm. Moreover, since def$\,(\cL - \lambda) =$ codim$\,(\cL - \lambda) > 0$ we have that $\ind(\cL - \lambda) \neq 0$.

Now suppose that $\lambda \in \sigma_\delta(\cL)$. Then from Lemma \ref{lemLsemi} we conclude that for every $0 < \epsilon \ll 1$ sufficiently small, $\cL^\epsilon - \lambda$ is semi-Fredholm with $\ind(\cL^\epsilon - \lambda) = \ind(\cL - \lambda) \neq 0$. This implies that $\lambda \not\in \C \backslash (-\infty, \mu_1]$. This shows that $\sigma_\delta(\cL) \subset (-\infty, \mu_1]$ and the theorem is proved.
\end{proof}

\subsection{Singular sequences and the location of $\sigma_\pi$}
In this section we control the remaining subset of the spectrum, namely, $\sigma_{\pi}(\cL)$. To this end, we use the technique developed in \cite{LeLoP22} and improve it by establishing a \emph{spectral gap}, an important property that could be used to study the nonlinear stability of the front. The idea is to prove that singular sequences (see Remark~\ref{remponla}) disperse the $L^2$-mass of perturbations to infinity, where the sign of the coefficients of $\cL$ is known. 
\begin{theorem}\label{lem:sd-sigmapi}
The $\sigma_{\pi}$-spectrum of the linearized operator $\cL$ around a stationary degenerated Nagumo front satisfies
$\sigma_{\pi}(\cL) \subset (-\infty, f'(1)]$.
\end{theorem}

\begin{proof}
Fix $\lambda \in \sigma_{\pi}(\cL)$. Then from Definition \ref{defspecd} the range of $\cL - \lambda$ is not closed and there exists a singular sequence $u_n \in \cD(\cL)$, 
with $\| u_n \|_{L^2} = 1$, for all $n \in \N$, such that $(\cL - \lambda) u_n \to 0$ in $L^2$ as $n \to \infty$ and which has no convergent subsequence. Since $L^2$ is a reflexive space, this sequence can be chosen so that it converges weakly to 0: $u_n \rightharpoonup 0$ in $L^2$ (see Remark \ref{remponla}).
Now, we can apply Lemma \ref{lem:core} to any fixed element $u_k \in \cD(\cL)$ and we obtain a sequence $v_{k,n}\in H^2(\R)$ such that
$$v_{k,n}\to u_k, \quad \mbox{ in } L^2 \qquad \mbox{ and } \qquad \cL v_{k,n}\to \cL u_k, \quad \mbox{ in } L^2,$$
as $n\to\infty$.
By using a standard diagonal argument, we can select a sequence $\tilde v_n\in H^2(\R)$ such that
$$\lim_{n\to\infty}\|\tilde v_n - u_n\|_{L^2}=0, \qquad \mbox{ and } \qquad \lim_{n\to\infty}\|\cL\tilde v_n - \cL u_n\|_{L^2}=0.$$
As a consequence, $\|(\cL -\lambda)\tilde v_n\|_{L^2}\to0$, as $n\to\infty$.
Moreover, the sequence $\tilde v_n$ can be chosen such that $\|\tilde v_n \|_{L^2} = 1$, for all $n \in \N$ and, clearly, $\tilde v_n \rightharpoonup 0$ in $L^2$.
Hence, without loss of generality, let us assume that the singular sequence $u_n \in H^2(\R)$.
Let us now prove that there exists a subsequence, still denoted by $u_n$, such that $u_n \to 0$ in $L^2_\textrm{loc}(-\infty,\omega_0)$ as $n \to \infty$. For that purpose, fix an open bounded interval $J \subset (-\infty, \omega_0) \subset \R$ and let $f_n := (\cL - \lambda) u_n$. 
Then,
\begin{equation*}
    	(a_2(x) \partial_x u_n)_x + a_1(x) \partial_x u_n + (a_0(x) - \lambda) u_n = f_n,
\end{equation*}
where, as before, we have denoted $a_2(x)= D(\varphi)$, $a_1(x) = D(\varphi)_x$ and $a_0(x) = D(\varphi)_{xx} + f'(\varphi)$. 
Notice that $a_2> 0$ in the interval $J$.
  
After taking the complex $L^2$-product of the latter equation with $u_n$ and integrating by parts, we obtain
\begin{equation}\label{eq:cruc-sing}
	-\int_\R a_2|\partial_x u_n|^2\,dx- \frac{1}{2} \int_\R\partial_x a_1|u_n|^2\,dx+ \int_\R a_0|u_n|^2\,dx - \lambda= \langle f_n, u_n \rangle_{L^2},
\end{equation}
where we used that $\|u_n\|_{L^2}=1$.
By taking the imaginary part of \eqref{eq:cruc-sing}, we end up with
\begin{equation*}
	\Im\lambda=-\langle f_n, u_n \rangle_{L^2},
\end{equation*}
and since
$$|\langle f_n, u_n \rangle_{L^2}|\leq\|f_n\|_{L^2}\to0, \qquad \mbox{ as } n\to\infty,$$
we conclude that $\lambda$ is real.
Moreover, by taking the real part of \eqref{eq:cruc-sing}, we deduce
\begin{equation}\label{eq:sd-lambda}
	\Re \lambda = -\int_\R a_2|\partial_x u_n|^2\,dx+\int_\R(a_0-\tfrac12\partial_x a_1)|u_n|^2\,dx- \Re \langle f_n, u_n \rangle_{L^2}.
\end{equation}
This equation together with the hypothesis on $u_n$, $f_n$, $\lambda$ fixed and the strict positivity of $a_2 = D(\varphi)$ on $J$, imply that $u_n$ is bounded in the space $H^1(J)$. Therefore, by the Rellich-Kondrachov Theorem there exists a subsequence, still denoted by $u_n$, such that $u_n \to 0$ in $L^2(J)$. By using a standard diagonal argument in increasing intervals, we arrive at a subsequence $u_n$ such that $u_n \to 0$ in $L^2_\textrm{loc}$, as $n \to \infty$.

On the other hand, observe that, as $f'(0)< 0$, for an arbitrary $\eta >0$, there exists $R>0$ big enough such that
\begin{equation*}
  	a_0(x) - \tfrac{1}{2} \partial_x a_1(x)  = \tfrac{1}{2} D(\varphi)_{xx} + f'(\varphi) < f'(1)+\eta, \quad \text{ for } x < -R.
\end{equation*}
and
\begin{equation*}
  	a_0(x) - \tfrac{1}{2} \partial_x a_1(x)  = \tfrac{1}{2} D(\varphi)_{xx} + f'(\varphi) < 0, \quad \text{ for } x \in[\omega_0-1/R,\omega_0].
\end{equation*}
Therefore, from $a_2(x) = D(\varphi) \geq 0$ and \eqref{eq:sd-lambda} we obtain
\[
\begin{aligned}
\Re \lambda &= -\int_{-\infty}^{\omega_0} a_2|\partial_x u_n|^2\,dx+\int_{-\infty}^{\omega_0}(a_0-\tfrac12\partial_x a_1)|u_n|^2\,dx- \Re \langle f_n, u_n \rangle_{L^2}\\
&\leq\int_{-\infty}^{-R} (a_0 - \tfrac{1}{2} \partial_x a_1) |u_n|^2 \, dx +\int_{-R}^{\omega_0-1/R} (a_0 - \tfrac{1}{2} \partial_x a_1) |u_n|^2 \, dx \\
&\qquad +\int_{\omega_0-1/R}^{\omega_0} (a_0 - \tfrac{1}{2} \partial_x a_1) |u_n|^2 \, dx +|\langle f_n, u_n \rangle_{L^2}| \\
&\leq f'(1) + \eta + \int_{-R}^{\omega_0-1/R} (a_0 - \tfrac{1}{2} \partial_x a_1) |u_n|^2 \, dx+\| (\cL - \lambda) u_n \|_{L^2} \to f'(1) + \eta,
\end{aligned}
\]
as $n \to \infty$, thanks to boundedness of the coefficients, $\| u_n \|_{L^2} = 1$, $(\cL - \lambda) u_n \to 0$ in $L^2$ and the convergence of $u_n$ to zero in $L^2_\textrm{loc}(-\infty,\omega_0)$. The result follows as $\eta>0$ is arbitrary.
\end{proof}

We now summarize the contents of Theorems \ref{thmpstab}, \ref{thmstabsdelta} and \ref{lem:sd-sigmapi} into the following result: the spectral stability, with a \emph{spectral gap} and real spectrum, of stationary degenerate Nagumo fronts.
\begin{theorem}[spectral stability]
\label{theostab}
Un\-der hy\-po\-theses \eqref{hypD} and \eqref{bistablef}, all stationary diffusion-degenerate Na\-gu\-mo fronts connecting the equilibrium states $u = 0$ and $u = 1$ are spectrally stable in the energy space $L^2(\R)$. More precisely, 
\[
\sigma(\cL) \subset (-\infty, -\beta] \cup \{0\},
\]
for a certain positive constant $\beta$. Moreover, $\lambda = 0$ is a simple eigenvalue associated to the eigenfunction $\varphi_x \in \cD(\cL)$ (translation invariance).
\end{theorem}
\begin{remark}
\label{nogapyet}
Notice that we have proved the existence of a \emph{positive spectral gap} (it suffices to choose $0 < \beta < \min \{\mu_0, \mu_1\}$), which is an important feature when studying the nonlinear stability of the fronts with semigroup methods.
\end{remark}

\section{Decay rates of the associated semigroup}
\label{secdecaysg}

In this section we establish the decaying properties of the semigroup generated by the linearized 
operator around the front. We need to show that $\cL$ is the infinitesimal generator of a semigroup and relate its decaying properties to its spectrum. 

\subsection{Resolvent estimate and generation of the semigroup}

From standard semigroup theory, we expect the linearized operator $\cL$ around a degenerate Nagumo front to generate a $C_0$-semigroup (see, for example, the results of Igari \cite{Iga73,Iga74a,Iga74b} for degenerate parabolic equations). It is remarkable, though, that the operator $\cL$ actually generates an \emph{analytic} semigroup. This fact follows immediately from the fact that the spectrum is real and bounded above (and hence, it can be confined to a sector in the complex space). For later use, however, we need to establish the following resolvent estimate, from which the analyticity of the semigroup can be deduced.
\begin{theorem}[resolvent estimate]
\label{lemresolvest}
There exists $\eta_0 > 0$ such that the set $\{\lambda \in \C \, : 
\, \Re \lambda > \eta_0\}$ is contained in $\rho(\mathcal{L})$ and
\begin{equation}
\label{resolvest}
\|(\lambda - \mathcal{L})^{-1}\|_{L^2 \to L^2} \leq \frac{2}{|\lambda - \eta_0|}, 
\end{equation}
for all $\Re \lambda > \eta_0$, where $\| \cdot \|_{L^2 \to L^2}$ denotes the operator norm in $L^2$.
\end{theorem}
\begin{proof}
First, notice that for any $\eta_0 > 0$ the half-plane $\{ \lambda \in \C \, : \, \Re \lambda > \eta_0 \}$ is contained in $\rho(\cL)$ because of the spectral stability result (Theorem \ref{theostab}).

Next, suppose that $\Re \lambda > 0$, $u\in H^2(\R)$, $(\lambda - \mathcal{L})u = g$, with some $g \in L^2(\R)$, and take the $L^ 2$-product of $u$ with the resolvent equation. 
This yields
\begin{equation}
\label{lauuno}
 \begin{aligned}
  \lambda \|u\|_{L^2}^2 &= \< u, (D(\varphi)u)_{xx} \>_{L^2} + \< u, f'(\varphi)u \>_{L^2} + \<u,g\>_{L^2}\\
  &= - \int_\R D(\varphi) |u_x|^2 \, dx - \int_\R D(\varphi)_x u_x^*u\,dx+ \int_\R f'(\varphi)|u|^2\,dx+\<u,g\>_{L^2},
 \end{aligned}
\end{equation}
after integration by parts. Since
\[
 \Re \int_\R D(\varphi)_x u_x^*u\,dx=\int_\R D(\varphi)_x \partial_x \left( \tfrac{1}{2}|u|^2\right) \, dx=
-\frac{1}{2} \int_\R D(\varphi)_{xx} |u|^2 \, dx,
\]
then taking the real part of \eqref{lauuno}, we obtain
\[
(\Re \lambda) \|u\|_{L^2}^2 = - \int_\R D(\varphi) |u_x|^2 \, dx + \int_\R q_0 |u|^2 \, dx + \Re \<u,g\>_{L^2},
\]
where 
\[
q_0 := \frac{1}{2} D(\varphi)_{xx} + f'(\varphi), \qquad x \in \R,
\]
is a real, bounded function. Therefore, if we denote $C_0 := \max_{x \in \R} |q_0(x)| > 0$ we obtain the estimate
\begin{equation}
\label{laccuatro}
(\Re \lambda) \|u\|_{L^2}^2 \leq - \int_\R D(\varphi) |u_x|^2 \, dx + C_0 \|u\|_{L^2}^2 + \|u\|_{L^2} \| g \|_{L^2}.
\end{equation}
Now, take the imaginary part of \eqref{lauuno}. The result is
\[
(\Im \lambda)  \|u\|_{L^2}^2 = -\Im\int_\R D(\varphi)_x u_x^*u\,dx+\Im \<u,g\>_{L^2}.
\]
Hence, for any $\vep > 0$ we have the estimate
\begin{equation}
\label{laccinco}
|\Im \lambda| \|u\|_{L^2}^2 \leq \|u\|_{L^2} \| g \|_{L^2} + \vep \| D(\varphi)_x u_x \|_{L^2}^2 + C_\vep \| u \|_{L^2}^2,
\end{equation}
for some constant $C_\vep > 0$. Combine estimates \eqref{laccuatro} and \eqref{laccinco} to arrive at
\begin{equation}
\label{lasseis}
\begin{aligned}
\Big( \Re \lambda + |\Im \lambda| \Big) \|u\|_{L^2}^2 \leq - &\int_\R \big( D(\varphi) - \vep D(\varphi)_x^2 \big) |u_x|^2 \, dx
\\ &\qquad + \big( C_0 + C_\vep \big) \|u\|_{L^2}^2 + 2 \|u\|_{L^2} \| g \|_{L^2}.
\end{aligned}
\end{equation}
Notice that since $D(\varphi) \equiv 0$ and $D'(\varphi) \varphi_x \equiv 0$ for $x \in (-\infty, \omega_0)$ then the integral on the right hand side of \eqref{lasseis} is reduced to
\[
\int_\R \big( D(\varphi) - \vep D(\varphi)_x^2 \big) |u_x|^2 \, dx = \int_{-\infty}^{\omega_0} \big( D(\varphi) - \vep D(\varphi)_x^2 \big) |u_x|^2 \, dx,
\]
for any $\vep > 0$. We claim that there exists $\vep\in(0,1)$ sufficiently small such that
\begin{equation}\label{poscoeff}
	D(\varphi) - \vep D(\varphi)_x^2= D(\varphi)\left(1-\vep\frac{D(\varphi)_x^2}{D(\varphi)}\right)\geq 0, \quad \text{for all } \; x \in (-\infty, \omega_0).
\end{equation}
Indeed, the function $D(\varphi)_x^2/D(\varphi)$ is positive and bounded, because
$$\lim_{x \to - \infty}\frac{D(\varphi)_x^2}{D(\varphi)}=\lim_{x\to-\infty}\frac{D'(\varphi)^2\varphi_x^2}{D(\varphi)}=0,$$ for any $\vep > 0$, 
and from \eqref{exp2} and \eqref{exp1}, it follows that
\[
\lim_{x \to \omega_0^{-}} \frac{D(\varphi)_x^2}{D(\varphi)}=-\frac{2D'(0)^2f'(0)}{3D'(0)^2}=-\frac{2}{3}f'(0) > 0.
\]
Hence, by virtue of $D(\varphi) > 0$, $|D'(\varphi) \varphi_x| > 0$ for all $x \in (-\infty, \omega_0)$ and by continuity, we conclude  that
\[
0 < M := \max_{x \in (-\infty, \omega_0]} \frac{D(\varphi)_x^2}{D(\varphi)}<\infty.
\]
Consequently, it suffices to select $\vep\in\left(0,\tfrac{1}{2M}\right)$ in order to obtain \eqref{poscoeff} which together with \eqref{lasseis} implies, in turn, the estimate
\[
\Big( \Re \lambda + |\Im \lambda| \Big) \|u\|_{L^2}^2 \leq \eta_0 \|u\|_{L^2}^2 + 2 \|u\|_{L^2} \| g \|_{L^2},
\]
where $\eta_0 := C_0 + C_\vep > 0$ is a fixed real positive constant. 
This last estimate implies that if $u\in H^2(\R)$ and $(\lambda-\mathcal{L})u=g$, then
\begin{equation}\label{eq:soloH2}
\| u \|_{L^2} \leq \frac{2 \| g \|_{L^2}}{|\lambda - \eta_0|}, \qquad \text{for all } \Re \lambda > \eta_0.
\end{equation}
Now, fix $\lambda\in \C$ with $\Re \lambda>\eta_0$ and fix $g\in L^2(\R)$.
Thus, there exists a unique $u\in\cD(\cL)$ such that $(\lambda-\mathcal{L})u=g$.
Let $u_n\in H^2(\R)$ be the sequence given by \eqref{eq:core}.
By applying \eqref{eq:soloH2} to $u_n$, we obtain
$$\| u_n \|_{L^2} \leq \frac{2 \| (\lambda-\cL)u_n \|_{L^2}}{|\lambda - \eta_0|}, \qquad \text{for all } \Re \lambda > \eta_0.$$
Passing to the limit and using \eqref{eq:core}, we conclude that
$$\|u\|_{L^2} \leq \frac{2 \| g \|_{L^2}}{|\lambda - \eta_0|}, \qquad \text{for all } \Re \lambda > \eta_0,$$
which implies the desired resolvent estimate \eqref{resolvest} and the proof is complete.
\end{proof}

An important consequence of Theorem \ref{lemresolvest} is the following result.
\begin{corollary}
\label{coranalytic}
The operator $\mathcal{L}$ is the infinitesimal generator of an analytic semigroup 
$\{e^{t\mathcal{L}}\}_{t\geq 0}$ in $L^2(\R)$, and its resolvent set $\rho(\mathcal{L})$ contains the 
sector 
$|\mathrm{arg}(\lambda - \eta_0)| < \pi/2 + \delta$ for some $\delta > 0$. Moreover, the semigroup is 
determined explicitly by Dunford's integral formula
\[
 e^{t\mathcal{L}} = \frac{1}{2\pi i} \int_\Gamma e^{\lambda t} (\lambda I - \mathcal{L})^{-1} \, d\lambda,
\]
where $\Gamma$ is any rectifiable curve in the complex plane from $\infty \cdot e^{-i\theta}$ to $\infty 
\cdot e^{i\theta}$, such that $\Gamma$ lies entirely in the set $\{|\mathrm{arg}(\lambda - \eta_0)| < 
\theta\}$ and $\theta$ is any angle satisfying $\pi/2 < \theta < \pi/2 + \delta$.
\end{corollary}
\begin{proof}
It follows from Theorem \ref{lemresolvest} and upon application of  the standard theory of generation of analytic semigroups (see, for example, Theorem 12.31 in \cite{ReRo04}).
\end{proof}
\begin{remark}
It is well-known that, for a parabolic differential operator of the form $\cA = D \partial_x^2 + a_1(x) 
\partial_x + 
a_0(x)$, where $D > 0$ is positive definite, if $a_1, a_0$ are smooth and decay exponentially fast then 
the operator $\cA$ is sectorial and the generator of an analytic semigroup (see \cite{KaPro13}, Section 4). 
This last result (Theorem \ref{lemresolvest} and Corollary \ref{coranalytic}) is remarkable in the sense that, 
for the stationary Nagumo front with degenerate diffusion, the associated linearized operator is also 
sectorial (and, hence, the infinitesimal generator of an analytic semigroup), even though the coefficients do 
not decay exponentially fast and the diffusion is not positive definite everywhere.
\end{remark}

\subsection{Semigroup decay}

By virtue of Corollary \ref{coranalytic}, from standard semigroup theory \cite{Pazy83}, we readily have that
\begin{equation*}
 \frac{d}{dt} \big(e^{t\mathcal{L}}u\big) = e^{t\mathcal{L}} (\mathcal{L} u) = \mathcal{L} 
(e^{t\mathcal{L}})u,
\end{equation*}
for all $u \in D(\mathcal{L})$.
Let us now recall the growth bound for a semigroup 
${\mathcal{S}(t)}$:
\begin{equation*}
 \omega(\mathcal{S}) = \inf \{\omega \in \R \, : \, \text{there exists} \; M 
\geq 1 \; \text{such that} \; \|{\mathcal{S}(t)}\| \leq M e^{\omega t} \; \forall t 
\geq 0\}.
\end{equation*}
We say a semigroup is \emph{uniformly (exponentially) stable} whenever $\omega(\cS) < 0$. If $\cL$ is the 
infinitesimal generator of a $C_0$-semigroup ${\mathcal{S}(t)}$ then its spectral bound is defined as 
\begin{equation*}
 s(\cL) := \sup \{ \Re \lambda \, : \, \lambda \in \sigma(\cL)\}.
\end{equation*}

Therefore, for the operator $\mathcal{L}$ under consideration, we may define the Hilbert space $X_1 \subset 
L^2(\R)$ as the range of the spectral projection,
\[
 \mathcal{P} u := u - \Theta^{-1}\<u,\psi\>_{L^2} \varphi_x, \qquad \; \; X_1 := \mathcal{R}(\mathcal{P}) 
\subset L^2(\R),
\]
where $\Theta$ is defined in \eqref{defTheta}. In this fashion, we project out the eigenspace spanned by the single eigenfunction $\varphi_x$. 
Outside this eigenspace, we expect the semigroup to decay exponentially as we shall verify in the sequel.

Since on a reflexive Banach space, weak and weak$^*$ topologies coincide, the family of dual 
operators $\{ \left({e^{t\mathcal{L}}}\right)^*\}_{t\geq 0}$, consisting of all the formal adjoints in $L^2$, 
is an analytic semigroup as well (cf. \cite{EN00}). Moreover, the infinitesimal generator of this semigroup 
is simply the formal adjoint $\mathcal{L}^*$ (see \cite[Corollary 10.6]{Pazy83}), justifying the 
following notation: $(e^{t \cL})^* = e^{t \cL^*}$ . Therefore, by semigroup theory, we have
\[
 e^{t\cL}\varphi_x = \varphi_x, \qquad e^{t\cL^*}\psi = \psi.
\]
Thus, ${e^{t\mathcal{L}}} {\mathcal{P}} = {\mathcal{P}} {e^{t\mathcal{L}}}$. This shows that $X_1$ is 
an ${e^{t\mathcal{L}}}$-invariant closed (Hilbert) subspace of $L^2(\R)$. So we define the 
domain
\begin{equation*}
 \widetilde{D}:= \{ u\in D \cap X_1 \, : \, \cL u\in X_1 \}
\end{equation*}
and the operator
\begin{equation*}
\left\{
\begin{aligned}
 \widetilde{\cL} &: \widetilde{D} \subset X_1 \to X_1,\\
 \widetilde{\cL} u &:= \cL u, \qquad u\in \widetilde{D},
 \end{aligned}
 \right.
\end{equation*}
as the restriction of $\cL$ on $X_1$, $\widetilde{\cL} := \cL_{|X_1}$. Therefore, $\widetilde{\cL}$ is a closed, densely defined 
operator on the Hilbert space $X_1$. Moreover, we observe that $\varphi_x \in \ker {\mathcal{P}}$. 
Hence, $\lambda = 0 \notin \ptsp(\widetilde{\cL})$. As a consequence of point spectral stability of 
$\cL$ (see Theorem \ref{theostab} above) we readily obtain 
\begin{equation*}
 \sigma( \widetilde{\cL} ) \subset (-\infty, -\beta],
\end{equation*}
with $\beta > 0$. By the above observations, we obtain the following result.
\begin{lemma}
\label{lemexpdecay}
The family of operators $\{\widetilde{\cS}(t)\}_{t\geq 0}$, $\widetilde{\cS}(t) : X_1 \to X_1$, defined as 
\begin{equation*}
 \widetilde{\cS}(t) u:= {e^{t\mathcal{L}}} ({\mathcal{P}} u), \qquad u\in X_1, \;\; t \geq 0,
\end{equation*}
is an analytic semigroup in the Hilbert space $X_1$ with infinitesimal generator $\widetilde{\cL}$. Moreover, this semigroup is uniformly exponentially stable and satisfies
\begin{equation}
\label{semigroupdecay}
\| \widetilde{\cS}(t) \| \leq M e^{-\beta t},
\end{equation}
for some $M \geq 1$.
\end{lemma}
\begin{proof}
The fact that $\widetilde{\cL}$ is the infinitesimal generator follows from the Corollary in Section 2.2 of 
\cite{EN00}. The semigroup properties are inherited from those of ${e^{t\mathcal{L}}}$ in 
$L^2(\R)$. Finally, since the spectral mapping theorem, namely, the property that 
\begin{equation}
\sigma({e^{t\mathcal{L}}}) \backslash 
\{0\} = e^{t\sigma(\cL)},
\end{equation} 
holds for analytic semi\-groups (see Corollary 2.10 in \cite{EN06}), 
then we know that 
\[
\omega(\widetilde{\cS}) = s(\widetilde{\cL}) = - \beta < 0,
\]
yielding the result.
\end{proof}

\section{Discussion}
\label{secdiscuss}

In this paper we have established the spectral stability of monotone stationary fronts for reaction diffusion-degenerate equations of Nagumo type. The stationary fronts, originally discovered by S\'anchez-Gardu\~no and Maini \cite{SaMa97}, are nearly sharp, in the sense that they arrive to the degenerate state at a finite point and the connection is of class $C^1$ but not of class $C^2$. As a consequence, the derivative of the profile, which is the usual translation eigenfunction, does not belong to $H^2(\R)$. This technical fact makes the spectral analysis of the linearized operator around the profile more complicated than the case of degenerate traveling fronts with positive speed \cite{LeLoP22}. We circumvented this technical difficulty by defining the operator on an appropriate dense domain and by establishing the energy estimates in the new domain. We proved that the spectrum of the linearized operator is real and stable. Moreover, we extended previous spectral stability results \cite{LeLoP22} for a subset of the approximate spectrum by showing the existence of a spectral gap, that is, a positive distance between the imaginary axis and the spectrum, with the exception of the origin (the eigenvalue associated to translation). The latter property yields exponential decay of the semigroup outside a finite-dimensional eigenspace.

This spectral information will be used in a key way in the forthcoming nonlinear stability analysis of these stationary fronts. The exponentially decaying semigroup and the simplicity of the translation eigenvalue will turn out to be essential in order to nonlinearly modulate perturbations of the stationary fronts via translations alone. These estimates and the dynamical equation for the translations will imply, in turn, the nonlinear asymptotic stability of the stationary degenerate Nagumo fronts. This analysis will be the content of a companion paper \cite{FHLP-II}.

\section*{Acknowledgements}
C. A. Hern\'andez Melo expresses his sincere gratitude to the Departamento de Matem\'aticas y Mec\'anica of the Instituto de Investigaciones en Matem\'aticas Aplicadas y en Sistemas of the Universidad Nacional Aut\'onoma de M\'exico, IIMAS-UNAM, for their hospitality and financial support during academic visits when this research was carried out. The work of R. Folino was partially supported by DGAPA-UNAM, program PAPIIT, grant IN-103425. L. F. L\'opez R\'ios was partially supported by DGAPA-UNAM, program PAPIIT, grant IN113225. The work of R. G. Plaza was partially supported by SECIHTI, M\'exico, grant CF-2023-G-122. 

%
%
%
%
%


\begin{thebibliography}{10}

\bibitem{AGJ90}
{\sc J.~Alexander, R.~Gardner, and C.~K. R.~T. Jones}, {\em A topological
  invariant arising in the stability analysis of travelling waves}, J. Reine
  Angew. Math. \textbf{410} (1990), pp.~167--212.

\bibitem{AlCa79}
{\sc S.~M. Allen and J.~W. Cahn}, {\em A microscopic theory for antiphase
  boundary motion and its application to antiphase domain coarsening}, Acta
  Metall. \textbf{27} (1979), no.~6, pp.~1085--1095.

\bibitem{Aron80}
{\sc D.~G. Aronson}, {\em Density-dependent interaction-diffusion systems}, in
  Dynamics and modelling of reactive systems ({P}roc. {A}dv. {S}em., {M}ath.
  {R}es. {C}enter, {U}niv. {W}isconsin, {M}adison, {W}is., 1979), W.~E.
  Stewart, W.~H. Ray, and C.~C. Conley, eds., vol.~44 of Publ. Math. Res.
  Center Univ. Wisconsin, Academic Press, New York-London, 1980, pp.~161--176.

\bibitem{Aron85}
{\sc D.~G. Aronson}, {\em The role of
  diffusion in mathematical population biology: {S}kellam revisited}, in
  Mathematics in biology and medicine ({B}ari, 1983), V.~Capasso, E.~Grosso,
  and S.~L. Paveri-Fontana, eds., vol.~57 of Lecture Notes in Biomath.,
  Springer, Berlin, 1985, pp.~2--6.

\bibitem{ARR81}
{\sc C.~Atkinson, G.~E.~H. Reuter, and C.~J. Ridler-Rowe}, {\em Traveling wave
  solution for some nonlinear diffusion equations}, SIAM J. Math. Anal.
  \textbf{12} (1981), no.~6, pp.~880--892.

\bibitem{BeOr78}
{\sc C.~Bender and S.~Orszag}, {\em Advanced Mathematical Methods for
  Scientists and Engineers}, McGraw-Hill, New York, 1978.

\bibitem{Bir97}
{\sc Z.~Bir{\'o}}, {\em Attractors in a density-dependent diffusion-reaction
  model}, Nonlinear Anal. \textbf{29} (1997), no.~5, pp.~485--499.

\bibitem{Bir02}
{\sc Z.~Bir{\'o}}, {\em Stability of
  travelling waves for degenerate reaction-diffusion equations of {KPP}-type},
  Adv. Nonlinear Stud. \textbf{2} (2002), no.~4, pp.~357--371.

\bibitem{ChIn74a}
{\sc N.~Chafee and E.~F. Infante}, {\em Bifurcation and stability for a
  nonlinear parabolic partial differential equation}, Bull. Amer. Math. Soc.
  \textbf{80} (1974), pp.~49--52.

\bibitem{ChIn74b}
{\sc N.~Chafee and E.~F. Infante}, {\em A bifurcation
  problem for a nonlinear partial differential equation of parabolic type},
  Applicable Anal. \textbf{4} (1974/75), pp.~17--37.

\bibitem{Cop78}
{\sc W.~A. Coppel}, {\em Dichotomies in Stability Theory}, no.~629 in Lecture
  Notes in Mathematics, Springer-Verlag, New York, 1978.

\bibitem{DaLoPe24}
{\sc A.-L. Dalibard, G.~L\'{o}pez-Ruiz, and C.~Perrin}, {\em Traveling waves
  for the porous medium equation in the incompressible limit: asymptotic
  behavior and nonlinear stability}, Indiana Univ. Math. J. \textbf{73} (2024),
  no.~2, pp.~581--643.

\bibitem{DiKa12}
{\sc J.~I. D{\'{\i}}az and S.~Kamin}, {\em Convergence to travelling waves for
  quasilinear {F}isher-{KPP} type equations}, J. Math. Anal. Appl. \textbf{390}
  (2012), no.~1, pp.~74--85.

\bibitem{EE87}
{\sc D.~E. Edmunds and W.~D. Evans}, {\em Spectral Theory and Differential
  Operators}, Oxford Mathematical Monographs, Clarendon Press, Oxford, 1987.

\bibitem{ElATa10a}
{\sc F.~El-Adnani and H.~Talibi-Alaoui}, {\em Traveling front solutions in
  nonlinear diffusion degenerate {F}isher-{KPP} and {N}agumo equations via the
  {C}onley index}, Topol. Methods Nonlinear Anal. \textbf{35} (2010), no.~1,
  pp.~43--60.

\bibitem{EN00}
{\sc K.-J. Engel and R.~Nagel}, {\em One-parameter semigroups for linear
  evolution equations}, vol.~194 of Graduate Texts in Mathematics,
  Springer-Verlag, New York, 2000.

\bibitem{EN06}
{\sc K.-J. Engel and R.~Nagel}, {\em A short course on
  operator semigroups}, Universitext, Springer-Verlag, New York, 2006.

\bibitem{Erde56}
{\sc A.~Erd\'elyi}, {\em Asymptotic expansions}, Dover Publications, Inc., New
  York, 1956.

\bibitem{FiM77}
{\sc P.~C. Fife and J.~B. McLeod}, {\em The approach of solutions of nonlinear
  diffusion equations to travelling front solutions}, Arch. Ration. Mech. Anal.
  \textbf{65} (1977), no.~4, pp.~335--361.

\bibitem{Fis37}
{\sc R.~A. Fisher}, {\em The wave of advance of advantageous genes}, Ann.
  Eugen. \textbf{7} (1937), pp.~355--369.

\bibitem{FHLP-II}
{\sc R.~Folino, C.~A. Hern\'{a}ndez~Melo, L.~F. L\'{o}pez~R\'{\i}os, and R.~G.
  Plaza}, {\em Stability of stationary reaction diffusion-degenerate {N}agumo
  fronts {II}: Nonlinear asymptotic stability}.
\newblock In preparation.

\bibitem{GiKe96}
{\sc B.~H. Gilding and R.~Kersner}, {\em A necessary and sufficient condition
  for finite speed of propagation in the theory of doubly nonlinear degenerate
  parabolic equations}, Proc. Roy. Soc. Edinburgh Sect. A \textbf{126} (1996),
  no.~4, pp.~739--767.

\bibitem{Hos86}
{\sc Y.~Hosono}, {\em Traveling wave solutions for some density dependent
  diffusion equations}, Japan J. Appl. Math. \textbf{3} (1986), no.~1,
  pp.~163--196.

\bibitem{Iga73}
{\sc K.~Igari}, {\em Cauchy problem for degenerate parabolic equations}, Proc.
  Japan Acad. \textbf{49} (1973), pp.~229--232.

\bibitem{Iga74b}
{\sc K.~Igari}, {\em Well-posedness of
  the {C}auchy problem for some evolution equations}, Publ. Res. Inst. Math.
  Sci. \textbf{9} (1973/74), pp.~613--629.

\bibitem{Iga74a}
{\sc K.~Igari}, {\em Degenerate parabolic
  differential equations}, Publ. Res. Inst. Math. Sci. \textbf{9} (1974),
  pp.~493--504.

\bibitem{I95}
{\sc G.~Iz{\'u}s, R.~Deza, O.~Ram{\'{\i}}rez, H.~S. Wio, D.~H. Zanette, and
  C.~Borzi}, {\em Global stability of stationary patterns in bistable
  reaction-diffusion systems}, Phys. Rev. E (3) \textbf{52} (1995), no.~1, part
  A, pp.~129--136.

\bibitem{Jerib15}
{\sc A.~Jeribi}, {\em Spectral theory and applications of linear operators and
  block operator matrices}, Springer-Verlag, Cham, 2015.

\bibitem{KaRo04a}
{\sc S.~Kamin and P.~Rosenau}, {\em Convergence to the travelling wave solution
  for a nonlinear reaction-diffusion equation}, Atti Accad. Naz. Lincei Cl.
  Sci. Fis. Mat. Natur. Rend. Lincei (9) Mat. Appl. \textbf{15} (2004),
  no.~3-4, pp.~271--280.

\bibitem{KaRo04b}
{\sc S.~Kamin and P.~Rosenau}, {\em Emergence of waves
  in a nonlinear convection-reaction-diffusion equation}, Adv. Nonlinear Stud.
  \textbf{4} (2004), no.~3, pp.~251--272.

\bibitem{KaPro13}
{\sc T.~Kapitula and K.~Promislow}, {\em Spectral and dynamical stability of
  nonlinear waves}, vol.~185 of Applied Mathematical Sciences, Springer, New
  York, 2013.

\bibitem{Kat80}
{\sc T.~Kato}, {\em Perturbation Theory for Linear Operators}, Classics in
  Mathematics, Springer-{V}erlag, {N}ew {Y}ork, {S}econd~ed., 1980.

\bibitem{KMMUS}
{\sc K.~Kawasaki, A.~Mochizuki, M.~Matsushita, T.~Umeda, and N.~Shigesada},
  {\em Modeling spatio-temporal patterns generated by {\em {b}acillus
  subtilis}}, J. Theor. Biol. \textbf{188} (1997), no.~2, pp.~177 -- 185.

\bibitem{KPP37}
{\sc A.~N. Kolmogorov, I.~Petrovsky, and N.~Piskunov}, {\em Etude de
  l’{\'e}quation de la diffusion avec croissance de la quantit{\'e} de
  matiere et son applicationa un probleme biologique}, Mosc. Univ. Bull. Math
  \textbf{1} (1937), pp.~1--25.

\bibitem{LeLoP22}
{\sc J.~F. Leyva, L.~F. L\'{o}pez~R\'{\i}os, and R.~G. Plaza}, {\em Spectral
  stability of monotone traveling fronts for reaction diffusion-degenerate
  {N}agumo equations}, Indiana Univ. Math. J. \textbf{71} (2022), no.~6,
  pp.~2335--2376.

\bibitem{LMP1}
{\sc J.~F. Leyva, C.~M\'{a}laga, and R.~G. Plaza}, {\em The effects of nutrient
  chemotaxis on bacterial aggregation patterns with non-linear degenerate cross
  diffusion}, Phys. A \textbf{392} (2013), no.~22, pp.~5644--5662.

\bibitem{LeP20}
{\sc J.~F. Leyva and R.~G. Plaza}, {\em Spectral stability of traveling fronts
  for reaction diffusion-degenerate {F}isher-{K}{P}{P} equations}, J. Dyn.
  Diff. Equ. \textbf{32} (2020), no.~3, pp.~1311--1342.

\bibitem{Lbr67a}
{\sc H.~M. Lieberstein}, {\em On the {H}odgkin-{H}uxley partial differential
  equation}, Math. Biosci. \textbf{1} (1967), no.~1, pp.~45--69.

\bibitem{McKe70}
{\sc H.~P. McKean, Jr.}, {\em Nagumo's equation}, Adv. Math. \textbf{4} (1970),
  pp.~209--223.

\bibitem{MeSc04a}
{\sc I.~Melbourne and G.~Schneider}, {\em Phase dynamics in the real
  {G}inzburg-{L}andau equation}, Math. Nachr. \textbf{263/264} (2004),
  pp.~171--180.

\bibitem{MurI3ed}
{\sc J.~D. Murray}, {\em Mathematical biology {I}. An introduction}, vol.~17 of
  Interdisciplinary Applied Mathematics, Springer-Verlag, New York, third~ed.,
  2002.

\bibitem{Muskat37}
{\sc M.~Muskat}, {\em The flow of homogeneous fluids through porous media},
  Mc-Graw-Hill, New York, 1937.

\bibitem{NAY62}
{\sc J.~Nagumo, S.~Arimoto, and S.~Yoshizawa}, {\em An active pulse
  transmission line simulating nerve axon}, Proc. IRE \textbf{50} (1962),
  no.~10, pp.~2061--2070.

\bibitem{New80}
{\sc W.~I. Newman}, {\em Some exact solutions to a nonlinear diffusion problem
  in population genetics and combustion}, J. Theor. Biol. \textbf{85} (1980),
  no.~2, pp.~325--334.

\bibitem{New83}
{\sc W.~I. Newman}, {\em The long-time
  behavior of the solution to a nonlinear diffusion problem in population
  genetics and combustion}, J. Theor. Biol. \textbf{104} (1983), no.~4,
  pp.~473--484.

\bibitem{OkLe01}
{\sc A.~Okubo and S.~A. Levin}, {\em Diffusion and ecological problems: modern
  perspectives}, vol.~14 of Interdisciplinary Applied Mathematics,
  Springer-Verlag, New York, second~ed., 2001.

\bibitem{Pazy83}
{\sc A.~Pazy}, {\em Semigroups of linear operators and applications to partial
  differential equations}, vol.~44 of Applied Mathematical Sciences,
  Springer-Verlag, New York, 1983.

\bibitem{ReRo04}
{\sc M.~Renardy and R.~C. Rogers}, {\em An introduction to partial differential
  equations}, vol.~13 of Texts in Applied Mathematics, Springer-Verlag, New
  York, second~ed., 2004.

\bibitem{SaMa94a}
{\sc F.~S{\'a}nchez-Gardu{\~n}o and P.~K. Maini}, {\em Existence and uniqueness
  of a sharp travelling wave in degenerate non-linear diffusion {F}isher-{KPP}
  equations}, J. Math. Biol. \textbf{33} (1994), no.~2, pp.~163--192.

\bibitem{SaMa95}
{\sc F.~S{\'a}nchez-Gardu{\~n}o and P.~K. Maini}, {\em Travelling wave
  phenomena in some degenerate reaction-diffusion equations}, J. Differ. Equ.
  \textbf{117} (1995), no.~2, pp.~281--319.

\bibitem{SaMa97}
{\sc F.~S{\'a}nchez-Gardu{\~n}o and P.~K. Maini}, {\em Travelling wave
  phenomena in non-linear diffusion degenerate {N}agumo equations}, J. Math.
  Biol. \textbf{35} (1997), no.~6, pp.~713--728.

\bibitem{San02}
{\sc B.~Sandstede}, {\em Stability of travelling waves}, in Handbook of
  dynamical systems, Vol. 2, B.~Fiedler, ed., North-Holland, Amsterdam, 2002,
  pp.~983--1055.

\bibitem{SMGA01}
{\sc R.~A. Satnoianu, P.~K. Maini, F.~S. Garduno, and J.~P. Armitage}, {\em
  Travelling waves in a nonlinear degenerate diffusion model for bacterial
  pattern formation}, Discrete Contin. Dyn. Syst. Ser. B \textbf{1} (2001),
  no.~3, pp.~339--362.

\bibitem{Sh10}
{\sc J.~A. Sherratt}, {\em On the form of smooth-front travelling waves in a
  reaction-diffusion equation with degenerate nonlinear diffusion}, Math.
  Model. Nat. Phenom. \textbf{5} (2010), no.~5, pp.~64--79.

\bibitem{SKT79}
{\sc N.~Shigesada, K.~Kawasaki, and E.~Teramoto}, {\em Spatial segregation of
  interacting species}, J. Theor. Biol. \textbf{79} (1979), no.~1, pp.~83 --
  99.

\bibitem{Ske51}
{\sc J.~G. Skellam}, {\em Random dispersal in theoretical poulations},
  Biometrika \textbf{38} (1951), no.~1-2, pp.~196--218.

\bibitem{Vaz07}
{\sc J.~L. V\'{a}zquez}, {\em The porous medium equation}, Oxford Mathematical
  Monographs, The Clarendon Press, Oxford University Press, Oxford, 2007.
\newblock Mathematical theory.

\bibitem{We10}
{\sc H.~Weyl}, {\em \"{U}ber gew\"ohnliche {D}ifferentialgleichungen mit
  {S}ingularit\"aten und die zugeh\"origen {E}ntwicklungen willk\"urlicher
  {F}unktionen}, Math. Ann. \textbf{68} (1910), no.~2, pp.~220--269.

\bibitem{XJMY24}
{\sc T.~Xu, S.~Ji, M.~Mei, and J.~Yin}, {\em Convergence to sharp traveling
  waves of solutions for {B}urgers-{F}isher-{KPP} equations with degenerate
  diffusion}, J. Nonlinear Sci. \textbf{34} (2024), no.~3, pp.~Paper No. 44,
  19.

\bibitem{XJMY25}
{\sc T.~Xu, S.~Ji, M.~Mei, and J.~Yin}, {\em Global stability of
  traveling waves for {N}agumo equations with degenerate diffusion}, J. Differ.
  Equ. \textbf{445} (2025), pp.~Paper No. 113587, 23.

\end{thebibliography}



\def\cprime{$'$}

%

\end{document}